\documentclass[12pt]{amsart}
\usepackage{amssymb}

\usepackage{hyperref}

\usepackage{mathrsfs}
\usepackage{bm}
\usepackage[all]{xy}

\allowdisplaybreaks[4]

\newtheorem{thm}{Theorem}[section]
\newtheorem{lem}[thm]{Lemma}
\newtheorem{conj}[thm]{Conjecture}

\newtheorem{cor}[thm]{Corollary}
\newtheorem{prop}[thm]{Proposition}

\def \para{\refstepcounter{thm} \par\medskip\noindent
                \textbf{\thethm .} }

\def \remark{\refstepcounter{thm} \par\medskip\noindent
                \textbf{Remark \thethm .} }

\numberwithin{equation}{thm}

\def\<{\langle}
\def\>{\rangle}

\newcommand{\ZZ}{\mathbb{Z}}

\newcommand{\cO}{\mathcal{O}}

\newcommand{\sH}{\mathscr{H}}

\newcommand{\fS}{\mathfrak{S}}
\newcommand{\ft}{\mathfrak{t}}

\newcommand{\ba}{\mathbf{a}}

\newcommand{\vG}{\varGamma}

\newcommand{\ra}{\rightarrow}
\newcommand{\LRa}{\Leftrightarrow}

\newcommand{\wt}{\widetilde}

\newcommand{\HInd}{\,^{\sH} \hspace{-1mm} \operatorname{Ind}}
\newcommand{\HRes}{\,^{\sH} \hspace{-1mm} \operatorname{Res}}
\newcommand\cmod{\operatorname{-mod}}

\newcommand{\Hom}{\operatorname{Hom}}
\newcommand{\Ext}{\operatorname{Ext}}
\newcommand{\End}{\operatorname{End}}
\newcommand{\proj}{\operatorname{-proj}}
\newcommand\ve{\varepsilon}
\newcommand{\fC}{\mathfrak{C}}
\newcommand{\fD}{\mathfrak{D}}
\newcommand{\bk}{\mathbf{k}}

\newcounter{ichi}
\newcommand{\roi}{\roman{ichi}}
\newcounter{ni}
\newcommand{\roii}{\roman{ni}}
\newcounter{san}
\newcommand{\roiii}{\roman{san}}
\newcounter{yon}
\newcommand{\roiv}{\roman{yon}}
\newcounter{go}
\newcommand{\rov}{\roman{go}}
\newcounter{roku}
\newcounter{nana}
\newcounter{hachi}
\newcounter{kyu}
\newcommand{\frh}{\mathfrak{h}}
\newcommand{\C}{\mathbb{C}}

\newcommand{\calD}{\mathcal{D}}
\newcommand{\sA}{\mathcal{A}}
\newcommand{\cS}{\mathcal{S}}
\DeclareMathOperator{\gEnd}{End}

\newcommand{\OInd}{\,^{\cO} \hspace{-0.5mm} \operatorname{Ind}}
\newcommand{\ORes}{\,^{\cO} \hspace{-0.5mm} \operatorname{Res}}
\DeclareMathOperator{\KZ}{KZ}
\DeclareMathOperator{\Id}{Id}
\DeclareMathOperator{\Image}{Im}

\begin{document}


\setlength{\baselineskip}{4.9mm}
\setlength{\abovedisplayskip}{4.5mm}
\setlength{\belowdisplayskip}{4.5mm}


\renewcommand{\theenumi}{\roman{enumi}}
\renewcommand{\labelenumi}{(\theenumi)}
\renewcommand{\thefootnote}{\fnsymbol{footnote}}
\renewcommand{\thefootnote}{\fnsymbol{footnote}}
\parindent=20pt


\setcounter{section}{-1}



\address{T. Kuwabara : Division of Mathematics, Faculty of Pure and Applied Sciences, 
University of Tsukuba, Tsukuba, Ibaraki 305-8571, Japan.} 
\email{kuwabara@math.tsukuba.ac.jp}

\address{H. Miyachi : Department of Mathematics, Osaka City University, 3-3-138 Sugimoto, Sumiyoshi-ku, Osaka 558-8585, Japan} 
\email{miyachi@sci.osaka-cu.ac.jp}

\address{K. Wada : Department of Mathematics, Faculty of Science, Shinshu University, 
		Asahi 3-1-1, Matsumoto 390-8621, Japan}
\email{wada@math.shinshu-u.ac.jp}



\medskip
\begin{center}
{\large \textbf{On  the Mackey formulas  for cyclotomic Hecke algebras 
	and categories $\cO$ of rational Cherednik algebras }}  
 \\
 \vspace{1cm}
Dedicated to Toshiaki Shoji on the occasion of his 70th birthday.   \\
 \vspace{1cm}
Toshiro Kuwabara, Hyohe Miyachi and Kentaro Wada 
\\[1em]
\end{center}


\title{} 
\maketitle 

\markboth{T.Kuwabara, H.Miyachi and K.Wada}{The Mackey formulas}



\begin{abstract}
In this paper, we shall establish the Mackey formulas in the following two set ups:
\begin{enumerate}
 \item  on the tensor induction and restriction functors on the modules over cyclotomic Hecke algebras (Ariki-Koike algebras) and their standard subalgebras of parabolic subgroups.
 \item  on the Bezrukavnikov-Etingof induction and restriction functors \cite{BE} among categories $\cO$ \cite{GGOR} of rational Cherednik algebras for
the complex reflection group of type $G(r,1,n)$ and their parabolic
		subgroups.
\end{enumerate}
\end{abstract}


\section{Introduction} 

The Mackey formula \cite{Mac},\cite[p.273]{CR} plays a very important role in representation theory
:
\begin{equation}\label{originalMackey}
{\mathrm{Res}}_L \circ {\mathrm{Ind}}^G (M) \cong
\bigoplus_{w \in L\backslash G / H}{\mathrm{Ind}}^L \circ {\mathrm{Res}}_{L \cap w H w^{-1}} (w \otimes M)
\end{equation}
for a finite group $G$, its subgroups $H,L$ and
$H$-module $M$.
In modular representation theory of finite groups, Green's vertex theory
is based on this formula [loc.cit].

In finite reductive groups, Dipper
and Fleischmann \cite[(1.14) Theorem]{MR1179779} established 
the Mackey formula on the Harish-Chandra induction and restriction for
Levi subgroups, and used it as an important base
 for their modular Harish-Chandra theory.
And, also in finite reductive groups,
the
Mackey formula on the Deligne-Lusztig induction and restriction has a
very important implication for the Lusztig conjecture on the characters
on these groups which is developed by C. Bonnaf\'e (see
\cite{Bon00}). This is an extension of Mackey formula on the
Harish-Chandra induction and restriction, although it is at the level of
characters.
So, the Mackey formula is subject to a subgroup lattice $\Lambda$ and a family of two kinds of functors ${\mathrm{Ind}}_A^B$ and $\mathrm{Res}_A^B$ labeled by the pairs $(A,B)$ with $A \subset B$
in this lattice $\Lambda$.

In this paper we shall report yet another Mackey formula for the case where $\Lambda$ is a set of parabolic subgroups of a complex reflection group.
More precisely, we shall tackle proving the following conjecture:
\begin{conj}[The Mackey formula for $\cO$]\label{ourconj}
For any finite complex reflection group $W$, and its parabolic subgroups $W_a$ and $W_b$,
 the Mackey formula with respect to the Bezrukavnikov-Etingof induction and restriction holds.
More precisely, at the level of representation categories, we have the
 following isomorphism of functors $:$
\begin{equation*}
  \ORes^W_{W_a} \circ \OInd^W_{W_b}
 \cong \bigoplus_{u \in {}^{a} W^{b}} 
\OInd^{W_a}_{W_a \cap u W_b u^{-1}} \circ\, u (-) \circ
 \ORes^{W_b}_{u^{-1} W_a u \cap W_b }, 
\end{equation*}
where 
${}^{a} W^{b}$ is a complete set of double coset representatives of $W_a \backslash W / W_b$.
\end{conj}
Here,
$\ORes^{W}_{W'}$ (resp. 
$\OInd^{W}_{W'}$
) is the Bezrukavnikov-Etingof
restriction (resp. induction) functor \cite{BE} 
and $u(-)$
 is the functor naturally induced by a conjugation (automorphism) by $u
 \in W$.
 
We write $W_{n,r}$ for the complex reflection group of type $G(r,1,n)$ in Shephard-Todd notation. We denote the set of standard generators of $W_{n,r}$ by $S=\{s_0,s_1,\ldots,s_{n-1}\}$
where $s_0$ (resp. $s_i$) has the order $r$ (resp. $2$),$s_0s_1s_0s_1=s_1s_0s_1s_0$ and $s_0s_j=s_js_0$ for $j>1$ and $s_1,\ldots, s_{n-1}$ satisfy the braid relation for the symmetric group $\fS_n$ on $n$ letters.
In this paper, we shall study the Mackey formula for the cyclotomic Hecke algebra $\sH_{n,r}=\<T_0,T_1,\ldots,T_{n-1}\>$ (see \ref{H Def} for the precise definition) 
and the categories $\cO$ of cyclotomic rational Cherednik algebras associated with $W_{n,r}$ (in so-called $t=1$ case).
 
In this paper, we shall establish the Mackey formulas in the following two set ups:
\begin{enumerate}
 \item  $\Lambda$ is the set of standard parabolic subgroups of  $W_{n,r}$ and 
${\mathrm{Ind}}_A^B$ is the tensor induction functor and ${\mathrm{Res}}_A^B$ is the restriction functor between Hecke algebras associated with $A$ and $B$ in $\Lambda$.

\item  $\Lambda$ is the set of parabolic subgroups of $W_{n,r}$.
The induction and restriction are 		
the Bezrukavnikov-Etingof induction $\OInd$ and restriction $\ORes$
		\cite{BE} respectively 
		among categories $\cO$ of cyclotomic rational Cherednik algebras for
the complex reflection group $W_{n,r}$ and their parabolic subgroups.
\end{enumerate}
The precise statement of (i) is
Theorem \ref{H Thm Mackey}.
The precise statement of (ii) is
Theorem \ref{thm:Mackey-RCA}, which supports Conjecture~\ref{ourconj}.
The part (i) is given in a characteristic-free manner, even holds over 
$\ZZ[q,q^{-1},Q_1, \dots,Q_r]$, where $q,Q_1,\dots,Q_r$ are indeterminate over $\ZZ$. 
On the contrary, the part
(ii) heavily depends on the coefficient field ${\mathbb{C}}$,
due to the use of KZ-functor, Riemann-Hilbert correspondence.
Next, we make remarks on the subgroup lattice $\Lambda$: Let $W$ be a complex
reflection group, and let $\frh$ be the reflection
$\C$-representation of $W$. By a parabolic subgroup of $W$, we mean a
stabilizer, in $W$, of some point in $\frh$.
We mean by a standard parabolic
subgroup of $W$ a special parabolic subgroup $\< I \>$ of $W$ for some
subset $I$ of the set of simple reflections.

Very briefly we remark some known results related to the above (i) and (ii):
In \cite[2.29]{MR948191}, the Mackey formula on the $1$-parameter Iwahori-Hecke algebras can be found.
In \cite{MR1925135}, the Mackey formula on the cyclotomic Hecke
algebras for the maximal co-rank 1 cases are treated, namely, it is with respect to
two identical
subgroups $W_a=W_{n-1,r}$ and $W_b=W_{n-1,r}$ of $W_{n,r}$.
Since in our set up we can take any two standard parabolic subgroups, the part (i) is a strong generalization of her result.
In \cite[Lemma 2.5]{SV}, at the level of the Grothendieck group the part
(ii) is considered. 
However, this is a consequence of Mackey's original formula (\ref{originalMackey}).
In \cite[Theorem 2.7.2]{LS}, they established Mackey formula for the categories
$\cO$ of rational Cherednik algebras or Coxeter groups.

In the case where $W$ is a finite Coxeter group, 
to obtain the Mackey formula for corresponding Hecke algebras, 
we discuss by using reduced expressions of group elements, the distinguished minimal coset representatives 
and their properties. 
However, 
in the case where $W$ is a complex reflection group which is not a Coxeter group, 
we have not enough properties for reduced expressions of group elements, 
and we do not know a good choice of coset representatives. 
These lacks of theory for complex reflection groups cause difficulty to obtain the Mackey formula for cyclotomic Hecke algebras. 
In this paper, we give a solution of this problem for  complex reflection groups of type $G(r,1,n)$.

Regarding on applications, 
as in first paragraph, the role of the Mackey formula in rational Cherednik algebras similar
to the one in \cite{MR1179779,MR1249581} is expected. And, an obvious
application is for a study of cohomology groups $\Ext_{\cO(W)}^i(\OInd^{W}_{W_a}(M),
\OInd^{W}_{W_b}(N))$ via Eckmann-Shapiro lemma
\begin{multline*}
 \Ext^i_{\cO(W)}(\OInd^{W}_{W_a}(M),
\OInd^{W}_{W_b}(N)) \\
\cong \bigoplus_{u \in {}^{a} W^{b}} 
\Ext^i_{\cO(W_a)}(
M,
\OInd^{W_a}_{W_a \cap u W_b u^{-1}} \circ\, u (-) \circ
 \ORes^{W_b}_{u^{-1} W_a u \cap W_b } N).
\end{multline*} 
Here, $\cO(W')$ is the category $\cO$ for a complex reflection group $W'$ defined in \cite{GGOR}. 
Especially, it is useful to study the endomorphism ring of an induced module.
For a parabolic subgroup $W_b$ of $W$ with $X$ being finite dimensional simple in $\cO(W_b)$,
the endomorphism ring $\End_{\cO(W)}(\OInd_{W_b}^W(X))$ is studied in \cite{LS}.
They call it a generalized Hecke algebra (see \cite[Theorem 3.2.4, Definition 3.2.5]{LS}).
Their strategy is very traditional like \cite{MR570873, MR716849}, 
but tactics is new, such as geometrical properties of the categories $\cO$.
In the case where $W$ is a Coxeter group, they obtained an explicit description of
the generalized Hecke algebra. In their argument, Mackey formula has an important role.

This paper is organized as follows. 

In \S \ref{symm}, we review some known facts on symmetric groups, 
and also give a technical result (Lemma \ref{S Lemma}) which shall be  used in \S \ref{cpxref}. 

In \S \ref{cpxref}, we shall determine a complete set
of representatives of double cosets $W_1\backslash W_{n,r} /W_2$ over two
standard parabolic subgroups $W_1,W_2$ of $W_{n,r}$. 
Throughout this paper, 
we use a expression of elements of $W_{n,r}$ being along  the semidirect product $W_{n,r} = \fS_n \ltimes (\ZZ/ r \ZZ)^n$. 
This expression will be used in \S \ref{H Section} 
to construct a basis of the cyclotomic Hecke algebra associated with $W_{n,r}$, so called Ariki-Koike basis. 
Our coset representatives are compatible with this expression, 
and they have a good behavior in the arguments for Hecke algebras. 
One of important property of our coset representatives is appeared in Proposition \ref{W Prop Wkupiu}. 
In this proposition, 
for each our representative $u$ of $W_1 \backslash W_{n,r} / W_2$, 
we prove that the subgroup $W_1 \cap u W_2 u^{-1}$ is a standard parabolic subgroup of $W_{n,r}$.  
Thanks to this fact, we can use Ariki-Koike basis in the discussion corresponding Hecke algebras. 
Another important property of our coset representatives is appeared in Proposition \ref{W prop only braid}. 
In this proposition, we prove that a certain identity of two expressions of a certain element of $W_{n,r}$ 
follows only from braid relations associated with $W_{n,r}$. 
This fact has important role to consider a certain twisting functor along module categories of Hecke algebras, 
and its lifting to categories $\cO$ of corresponding rational Cherednik algebras as explaining below.
We remark that our coset representatives are not the distinguished minimal coset representatives 
in the case where $r=2$ (i.e. $W_{n,r}$ is  Weyl group of type $B_n$). 
Thus, our representatives are not a generalization 
of  the distinguished minimal coset representatives for finite Coxeter groups. 

In \S \ref{H Section}, 
we shall establish the Mackey formula for cyclotomic Hecke algebras of type $G(r,1,n)$ in Theorem \ref{H Thm Mackey}. 
The most important key fact  is the existence of specific
element $T_u$ satisfying Proposition~\ref{H Prop TuHmnu},
 which comes from the construction of a specific representative of double
 coset $W_1 u W_2$ satisfying Proposition~\ref{W prop only braid}.
 
In \S \ref{category O}, 
we review some known facts on categories $\cO$ of rational Cherednik algebras associated with any complex reflection groups,  
and also give some technical results.  
Then we give a lifting argument from the Mackey formula for cyclotomic Hecke algebras 
to one for categories $\cO$ of rational Cherednik algebras 
under some technical assumptions 
in Proposition \ref{prop:H-Mackey-RCA}. 
This lifting argument has been employed in \cite[Theorem 2.7.2]{LS} for
Coxeter group case.

In \S \ref{SectMackey-O-Wnr}, 
we shall establish the Mackey formula for categories $\cO$ of rational Cherednik algebras of type $G(r,1,n)$ in Theorem \ref{thm:Mackey-RCA}. 
We employ the identity given in Proposition \ref{W prop only braid}, which follows only from braid relations, 
to show the correspondent of the functor $T_u(-)$ in the rational Cherednik algebra is the twisting functor $u (-)$
as a lift. This is recorded in Proposition \ref{prop:twist-KZ}. 
Then we can apply the lifting argument given in Proposition \ref{prop:H-Mackey-RCA} 
to obtain the Mackey formula for categories $\cO$.

In Appendix A, 
a general lifting argument on functors through the double centralizer properties 
is recorded, which
will be used \S 4.

In Appendix B, we compare known results on the coset representatives in
\cite{RS}  with the ones in \S 2 for some special cases (i.e. the case where $\mu=(1^{n-l})$ and $\nu=(1^{n-m})$) as a sort of independent interest. 
We remark that the coset representatives in \cite{RS} follows from notion of root systems of type $G(r,1,n)$. However, in \cite{RS}, they give the coset representatives only for special cases 
where $\mu=(1^{n-l})$ and $\nu=(1^{n-m})$, 
and we do not know whether we can obtain the coset representatives by using root systems in general. 
We also remark that our coset representatives are not generalization of ones in \cite{RS} 
(see Remark \ref{Remark root system}). 
For the reader being only interested in the Mackey
formula, he or she can skip this appendix. 
\\

{\bf Acknowledgements:} 
The authors would like to thank Seth Shelley-Abrahamson for letting us notice their paper \cite{LS}.
There were some overlaps between {\em loc. cit.} and the original version of this paper.
The first author was supported by JSPS KAKENHI Grant Number JP17K14151. 
The third author was supported by JSPS KAKENHI Grant Number JP16K17565. 



\section{The symmetric groups} \label{symm}
In this section, we review some known results on symmetric groups 
which follow from the general theory of Coxeter groups 
(see e.g. \cite{H}, \cite[Chapter 4]{DDPW}) 
except \eqref{S-row-std} and Lemma \ref{S Lemma}. 


\para 
\label{S Def}
Let $\fS_n$ be the symmetric group on $n$ letters. 
We consider the natural left action of $\fS_n$ on $\{1,2,\ldots,n\}$.
So, when $x \in \fS_n$ sends $i$ to $j$, 
we denote it by $x(i)=j$. 
For $i=1,2, \ldots, n-1$, 
let $s_i = (i,i+1)$ be the adjacent transposition. 
Then $S=\{s_1,s_2,\dots, s_{n-1}\}$ is a set of simple reflections of $\fS_n$. 
For $x \in \fS_n$, we denote the length of $x$ by $\ell(x)$. 
We denote the Bruhat order on $\fS_n$ by $\geq$. 

For integers $k_1 \leq k_2 \in \ZZ$, 
we denote the integer interval $\{k_1, k_1+1,\dots, k_2\}$ by $[k_1,k_2]$. 
For $1 \leq k_1 \leq k_2 \leq n$, 
we denote by $\fS_{[k_1, k_2]}$ 
the subgroup of $\fS_n$ generated by $\{s_{k_1}, s_{k_1 +1}, \dots, s_{k_2-1}\}$, 
namely $\fS_{[k_1,k_2]}$ is the subgroup permuting the set $\{k_1, k_1+1,\dots, k_2\}$. 

A composition of $n$ is a sequence of non-negative integers $\mu=(\mu_1, \mu_2, \dots )$ 
such that $\sum_{i} \mu_i =n$, 
and we denote it by $\mu \vDash n$. 
We also denote $|\mu|= \sum_{i}\mu_i$. 

For $\mu=(\mu_1, \mu_2,\dots, \mu_l) \vDash n$, 
let $\fS_{\mu}$ be the parabolic subgroup of $\fS_n$ associated with $\mu$, 
namely $\fS_{\mu}$ is the subgroup of $\fS_n$ generated by 
\begin{align*} 
S_{\mu} := S \setminus \{s_j \mid j= \sum_{i=1}^{k} \mu_i \text{ for some } k \geq 1\}. 
\end{align*}
We have 
$\fS_{\mu} \cong \fS_{\mu_1} \times \dots \times\fS_{\mu_l}$. 
For $\mu \vDash n$, put 
\begin{align*}
& \fS^\mu =\{ x \in \fS_n \mid \ell (x s ) > \ell (x) \text{ for all } s \in S_{\mu} \}, 
\\
& ^\mu \fS = \{ x \in \fS_n \mid \ell ( s x  ) > \ell (x) \text{ for all } s \in S_{\mu} \}, 
\end{align*}
then 
$\fS^\mu$ (resp. $^\mu \fS$) is the set of distinguished coset representatives 
of the coset $\fS_n/\fS_{\mu}$ (resp. $\fS_{\mu} \backslash \fS_n$).  
In particular, we have 
\begin{align}
\begin{split} 
&\ell( x y ) = \ell (x) + \ell (y) \text{ for } x \in \fS^{\mu}, y \in \fS_{\mu}, 
\\ 
&\ell(x y) = \ell (x) + \ell(y) \text{ for } y \in \, ^\mu \fS, x \in \fS_{\mu}. 
\end{split}
\end{align}
For $\mu, \nu \vDash n$, 
put 
${}^\mu \fS^\nu = {}^\mu \fS \cap \fS^{\nu}$, 
then ${}^{\mu} \fS^{\nu}$ is a complete set of representatives of the double cosets 
$\fS_{\mu} \backslash \fS_n / \fS_{\nu}$.  

For $x \in \, ^\mu \fS^\nu$, 
let $\tau(x) \vDash n$ be the composition determined by the equation 
$S_{\tau(x)} = S_{\mu} \cap x S_{\nu} x^{-1}$.
Then it is known that $\fS_{\mu} \cap x \fS_\nu x^{-1}$ is generated by $S_{\tau(x)}$.  
In particular, we have 
$\fS_{\mu} \cap x \fS_\nu x^{-1} = \fS_{\tau(x)}$.
By the general theory of Coxeter groups, 
we see that 
$w \in \fS_n$ 
is uniquely written 
as $w = y x z$ ($x \in \, ^{\mu} \fS^{\nu}$, $y \in (\fS_{\mu})^{\tau(x)}$, $z \in \fS_{\nu}$), 
and 
we have 
\begin{align}
\label{S length doublecoset decom}
\ell(y x z) 
=\ell (y) + \ell (x) + \ell (z) 
\quad 
(x \in \, ^{\mu} \fS^{\nu}, \, y \in (\fS_{\mu})^{\tau(x)}, \, z \in \fS_{\nu}). 
\end{align}

\para 
The distinguished coset representatives 
$\fS^\mu$ (resp. $^\mu \fS$) 
is described by a standard combinatorics as follows. 
For $\mu \vDash n$, 
the diagram of $\mu$ is the set 
$[\mu]=\{(i, j) \in \ZZ_{\geq 0}^2 \mid i \geq 1, 1 \leq j \leq
\mu_i\}$.
 Here, we take the English fashion for treating the element of $[\mu]$,
 for example, we say that there are $\mu_i$ boxes in the $i$-th row of
 $[\mu]$, we also say that $(i,1)$ is the left most box of the $i$-th row if
 $(i,1) \in [\mu]$, etc.
For $\mu \vDash n$, a $\mu$-tableau is a bijection $\ft : [\mu] \ra \{1,2, \dots, n\}$. 
The symmetric group $\fS_n$ acts on the set of $\mu$-tableaux from left 
by permuting the entries inside a given tableau, 
namely, for $x \in \fS_n$ and $\mu$-tableau $\ft$, the defining equation
is 
\begin{align*}
(x \cdot \ft) (i,j) = x ( \ft(i,j)) \quad ((i,j) \in [\mu]).
\end{align*}

We say that a $\mu$-tableau $\ft$ is row-standard 
if $\ft(i,j) < \ft(i, j+1)$ for all $(i,j) \in [\mu]$ such that $(i,j+1) \in [\mu]$, 
namely if the entries in $\ft$ increase from left to right in each row. 

For $\mu \vDash n$, 
let $\ft^{\mu}$ be the $\mu$-tableau in which the integers $1,2,\dots, n$, 
are attached in the order from left to right and from top to bottom in $[\mu]$, 
namely we have 
\begin{align*}
\ft^{\mu} (i,j) = \sum_{k=1}^{i-1} \mu_k + j 
\quad ((i,j) \in [\mu]). 
\end{align*}
Then, we have 
\begin{align}
\label{S-row-std}
\begin{split}
&\fS^\mu = \{ x  \in \fS_n \mid x \cdot \ft^{\mu} \text{ is row-standard} \}, 
\\
&^\mu \fS = \{ x \in \fS_n \mid x^{-1} \cdot \ft^{\mu} \text{ is row-standard} \} 
\end{split}
\end{align}
(see \cite[Proposition 3.3]{M}). 

For a convenience in later arguments, 
for $0\leq l \leq n$ and $\mu \vDash n -l$, 
we put $(l, \mu)=(l, \mu_1, \mu_2, \dots) \vDash n$. 
Then, we have the following lemma: 

\begin{lem}
\label{S Lemma}
For $x \in \,^{(l,\mu)} \fS^{(m, \nu)}$ for some $0 \leq l,m \leq n$, $\mu \vDash n-l$, $\nu \vDash n-m$, 
put 
\begin{align*}
c= \min\{ i \geq 0 \mid x(i+1) \not= i+1 {\mbox{ or }} i=n\} 
\text{ and } 
k=\min\{c,l,m\}.
\end{align*}
Then we have 
$x \in \fS_{[k+1,n]}$ and $[1,l] \cap \{x(1), x(2), \dots, x(m)\} =[1, k]$. 
\end{lem}

\begin{proof} 
If $c \geq \min \{l,m\}$, it is clear. 
Suppose $c < \min\{l,m\}$ (note that $k=c$ in this case), 
we have 
\begin{align}
\label{S-xc}
x(c)=c < c+1 < x (c+1) < x(c+2) < \dots < x (m) 
\end{align}
and there exists $b > m$ such that $x(b) =c+1$ 
since $x \cdot \ft^{(m,\nu)}$ is row-standard by \eqref{S-row-std}. 

If $x (c+1)> l $, 
we have $x \in \fS_{[c+1,n]}$ and $[1, l] \cap \{x(1),  \dots, x(m)\}= [1, c]$ 
by \eqref{S-xc}.  

If $x (c+1) \leq l$, we have $c+1 < x (c+1) \leq l$. 
Then we see that both $c+1$ and $x (c+1)$ appear in the first row of $\ft^{(l,\mu)}$. 
On the other hand, we have 
\begin{align*}
x^{-1}(c+1) =b >m \geq c+1 = x^{-1} (x (c+1)).
\end{align*} 
This contradicts that $x^{-1} \cdot \ft^{(l,\mu)}$ is row-standard. 
Thus this case does not occur.   
\end{proof}



\section{The complex reflection group of type $G(r,1,n)$} \label{cpxref}
In this section, 
we study on the complex reflection group $W_{n,r}$ of type $G(r,1,n)$. 
For  standard parabolic subgroups $W_{(l,\mu)}$ and $W_{(m,\nu)}$ of $W_{n,r}$, 
we shall find a complete set of representatives of the cosets $W/ W_{(l,\mu)}$ 
and the double cosets $W_{(l,\mu)} \backslash W_{n,r} / W_{(m,\nu)}$. 
These representatives will be used in the next section 
to obtain the Mackey formula for cyclotomic Hecke algebras.

\para \label{W Def}
The complex reflection group of type $G(r,1,n)$ is the semidirect product 
$W_{n,r} = \fS_n \ltimes (\ZZ / r \ZZ)^n$, 
where $\fS_n$ acts on $(\ZZ/ r \ZZ)^n$ via the permutation of factors.  
The group $W_{n,r}$ has a presentation 
such that $W_{n,r}$ is generated by  $s_0, s_1, \dots, s_{n-1}$ subject to the defining relations 
\begin{align*}
&s_0^r=1, \, 
s_i^2=1 \, (1\leq i \leq n-1), 
\\
&s_0 s_1 s_0 s_1 = s_1 s_0 s_1 s_0, \, 
s_i s_{i+1} s_i = s_{i+1} s_i s_{i+1} (1\leq i \leq n-2), \, 
s_i s_j = s_j s_i \, (|i-j|>1).
\end{align*}
The relations in the second row are called the braid relations.
Put 
\[
 t_i=s_{i-1} s_{i-2} \dots s_1 s_0 s_1 \dots s_{i-2} s_{i-1}
\] 
for $i=1$, $2$, $\dots$, $n$. 
Then $S=\{s_1,s_2, \dots, s_{n-1}\}$ generates $\fS_n$, 
and $t_i$ generates $\ZZ / r\ZZ$, the $i$-th factor of $(\ZZ/ r \ZZ)^n$. 
Then we have 
\begin{align*}
W_{n,r} = \{ x t_1^{a_1} t_2^{a_2} \dots t_n^{a_n} \mid x \in \fS_n, \, 
	a_1, a_2, \dots, a_n \in [0,r-1] \}. 
\end{align*}

From the definitions,  we have 
\begin{align}
\label{W-rel}
\begin{split} 
& t_i t_j = t_j t_i \quad (1 \leq i,j \leq n), 
\\
&x t_i x^{-1} = t_{x(i)} \quad (x \in \fS_n, \, 1 \leq i \leq n). 
\end{split}
\end{align}   


\para \label{W Def par}
For $0 \leq l \leq n$ and $\mu \vDash n-l$, 
let $W_{(l,\mu)}$ be the subgroup of $W_{n,r}$ generated by 
\begin{align*}
X_{(l,\mu)} = \{s_0,s_1,\dots, s_{n-1}\} \setminus \{ s_j \mid j=l+\sum_{i=0}^k \mu_i \text{ for some } k \geq 0\}, 
\end{align*}
where we put $\mu_0=0$. 
It is well-known that any parabolic subgroup of $W_{n,r}$ is 
conjugate to $W_{(l,\mu)}$ for some $0 \leq l \leq n$ and $\mu \vDash n-l$. 

Put 
\begin{align*}
S_{(l,\mu)} = X_{(l, \mu)} \cap S, 
\quad 
S_{(l)} =\{s_1, \dots, s_{l-1}\}, 
\quad 
S_{\mu}^{[l]} = S_{(l,\mu)} \setminus S_{(l)},
\end{align*}
where we put $S_{(l)} = \emptyset$ if $l\leq 1$. 
We easily see that 
\begin{itemize} 
\item 
the subgroup generated by $\{s_0,s_1,\dots, s_{l-1}\}$ is $W_{l,r}$, 
where we put $W_{l,r}=1$ if $l=0$, 
\item 
the subgroup generated by $S_{(l,\mu)}$ (resp. $S_{(l)}$) 
is the parabolic subgroup $\fS_{(l,\mu)}$ (resp. $\fS_{(l)}$)   of $\fS_n \subset W_{n,r}$ 
associated with $(l,\mu)$ (resp. $(l)$), 
\item 
the subgroup generated by $S_{\mu}^{[l]}$ is the parabolic subgroup $\fS_{\mu}^{[l]}$ of 
$\fS_{[l+1,n]}$ associated with $\mu$.  
\end{itemize}
Note that $\fS_\mu^{[l]}$ is contained in the centralizer of $W_{l,r}$,
we have 
\begin{align*} 
W_{(l,\mu)} = W_{l,r} \times \fS_{\mu}^{[l]} \cong (\fS_l \ltimes (\ZZ/ r \ZZ)^l) \times \fS_{\mu}, 
\end{align*} 
and   
\begin{align*}
W_{(l,\mu)} = \{x t_1^{a_1} t_2^{a_2} \dots t_l^{a_l} \mid x \in \fS_{(l, \mu)}, \, 
	 a_1, a_2,\dots, a_l \in [0,r-1] \}. 
\end{align*}
Put 
\begin{align*}
& W^{(l,\mu)} 
 = \{ x t_{l+1}^{a_{l+1}} t_{l+2}^{a_{l+2}} \dots t_n^{a_n} 
 	\mid x \in \fS^{(l,\mu)} , \,  a_{l+1}, a_{l+2}, \dots, a_n \in [0,r-1] \}, 
\\
& ^{(l,\mu)} W 
	=\{t_n^{a_n} \dots t_{l+2}^{a_{l+2}} t_{l+1}^{a_{l+1}} x 
	\mid x \in \, ^{(l,\mu)} \fS, \, a_{l+1}, a_{l+2}, \dots, a_n \in [0,r-1] \}. 
\end{align*}

 
\begin{lem} 
\label{W Lemma coset rep} 
The set $W^{(l,\mu)}$ $($resp. $^{(l,\mu)} W$ $)$ is a complete set of representatives 
of the coset  
$W_{n,r} / W_{(l,\mu)}$ $($resp. $W_{(l,\mu)} \backslash W_{n,r})$. 
\end{lem}

\begin{proof} 
We prove only the claim for $W^{(l,\mu)}$ since the claim for $^{(l,\mu)}W$ is proven in a similar way. 

For $w = x t_1^{a_1} t_2^{a_2} \dots t_n^{a_n} \in W_{n,r}$ ($x \in \fS_n$, $ a_1,\dots, a_n \in [0,r-1] $),  
we can write 
\begin{align*} 
& x = x_1 x_2 \quad (x_1 \in \fS^{(l,\mu)}, \, x_2 \in \fS_{(l,\mu)}) 
\text{ and } 
x_2 = y_1 y_2  \quad (y_1 \in \fS_{(l)}, \, y_2 \in \fS_{\mu}^{[l]}). 
\end{align*} 
Note that 
$y_1 \in \fS_{[1,l]}$ and $ y_2 \in \fS_{[l+1,n]}$, 
the relations \eqref{W-rel} imply that 
\begin{align*}
w 
&= x t_1^{a_1} t_2^{a_2} \dots t_n^{a_n}
\\
&= x_1 y_1 y_2 t_1^{a_1} t_2^{a_2} \dots t_n^{a_n}
\\
&=x_1 (t_{y_1(1)}^{a_1} \dots t_{y_1(l)}^{a_l}) (t_{y_2(l+1)}^{a_{l+1}} \dots t_{y_2(n)}^{a_n}) y_1 y_2 
\\
&= x_1 (t_{y_2(l+1)}^{a_{l+1}} \dots t_{y_2(n)}^{a_n}) (t_{y_1(1)}^{a_1} \dots t_{y_1(l)}^{a_l}) y_1 y_2  
\\
&= x_1 (t_{y_2(l+1)}^{a_{l+1}} \dots t_{y_2(n)}^{a_n}) y_1 y_2 (t_1^{a_1} \dots t_l^{a_l}) 
\\
&= x_1 (t_{y_2(l+1)}^{a_{l+1}} \dots t_{y_2(n)}^{a_n}) x_2 (t_1^{a_1} \dots t_l^{a_l}), 
\end{align*}
and we see that 
$x_1(t_{y_2(l+1)}^{a_{l+1}} \dots t_{y_2(n)}^{a_n}) \in W^{(l,\mu)}$ 
and 
$ x_2 (t_1^{a_1} \dots t_l^{a_l}) \in W_{(l,\mu)}$. 
Thus, we have 
\begin{align}
\label{W-coset decom}
W_{n,r} = \bigcup_{ u \in W^{(l,\mu)}} u W_{(l,\mu)}.
\end{align}
On the other hand, 
we see that 
$|W_{n,r}|=|\fS_n| r^n$, $|W_{(l,\mu)}| = |\fS_{(l,\mu)}| r^l$ and $|W^{(l,\mu)}| = |\fS^{(l,\mu)}| r^{n-l}$, 
and we have  
\begin{align}
\label{W-index}
[W_{n,r} : W_{(l,\mu)}] 
= |W_{n,r}| / |W_{(l,\mu)}| 
= (|\fS_n| / |\fS_{(l,\mu)}|) r^{n-l} 
= |\fS^{(l,\mu)}| r^{n-l} 
=|W^{(l,\mu)}|. 
\end{align}
Thus \eqref{W-coset decom} and \eqref{W-index} imply the claim for $W^{(l,\mu)}$. 
\end{proof} 


\remark 
In the case where $r=2$, 
the group $W_{n,2}$ is the Weyl group of type $B_n$. 
In this case, 
$W^{(l,\mu)}$ (resp. $^{(l,\mu)} W$) is not the set of distinguished coset representatives in general. 
For an example, 
take $l=0$ and $\mu \vDash n$ such that $\mu_1 >2$. 
Then 
$W_{(0,\mu)}$ is generated by $S_{\mu}$. 
In this case, $s_1 \in S_{\mu}$, 
and $t_2 \in W^{(l,\mu)}$. 
However, 
we have $\ell (t_2) =\ell (s_1 s_0 s_1) =3$ and $\ell (t_2 s_1) = \ell (s_1 s_0)=2$. 
Thus, $t_2$ is not a distinguished coset representative. 


\para \label{W Def Ix}
For $x \in \, ^ {(l,\mu)} \fS^{(m,\nu)}$ ($0 \leq l,m \leq n$, $\mu \vDash n-l$, $\nu \vDash n-m$), 
put 
\begin{align*}
I (x) = [m+1,  n] \cap \{x^{-1}(l+1), x^{-1}(l+2),\dots, x^{-1}(n)\}.
\end{align*}

For $x t_{m+1}^{a_{m+1}} \dots t_{n}^{a_n} \in W^{(m,\nu)}$ ($x \in \fS^{(m,\nu)}$), 
we have 
$x t_{m+1}^{a_{m+1}} \dots t_{n}^{a_n} 
=  t_{x(m+1)}^{a_{m+1}} \dots t_{x(n)}^{a_n} x$
by \eqref{W-rel}. 
Thus we see that 
$x t_{m+1}^{a_{m+1}} \dots t_{n}^{a_n} \in \, ^{(l,\mu)} W \cap W^{(m,\nu)}$
if and only if
$x \in \,^{(l,\mu)} \fS^{(m,\nu)} \text{ and } x(k) \in [l+1,n]$ with $a_k \ne 0$.
This implies that 
\begin{align*}
^{(l,\mu)} W \cap W^{(m,\nu)} 
=\{ x \prod_{i \in I (x)} t_i^{a_i} \mid x \in \,^{(l,\mu)} \fS^{(m,\nu)}, \,  a_i \in [0,r-1]\}.
\end{align*}

For $x \in \,^{(l,\mu)} \fS^{(m,\nu)}$, 
recall that  $\tau(x)$ is the composition such that 
\begin{align}
\label{W Def taux}
S_{\tau(x)} = S_{(l,\mu)} \cap x S_{(m,\nu)} x^{-1}, 
\end{align}
and we have 
$\fS_{\tau(x)} = \fS_{(l,\mu)} \cap x \fS_{(m,\nu)} x^{-1}$. 

For $z = x y x^{-1} \in \fS_{\tau(x)}$ ($z \in \fS_{(l,\mu)}, y \in \fS_{(m,\nu)}$),
we see that 
$y(i) \in [m+1, n]$ if $i \in [m+1,  n]$
since $y \in \fS_{(m,\nu)}$.  
We also see that 
\begin{align*}
y (i) \in \{x^{-1} (l+1),  \dots, x^{-1}(n) \} 
\text{ if }
i \in \{x^{-1} (l+1), \dots, x^{-1}(n) \} 
\end{align*}
since 
$y x^{-1}(l+ j) = x^{-1} z (l+j)$ and $z \in \fS_{(l,\mu)}$. 
These imply that 
\begin{align}
\label{W-y I}
y (i) \in I(x) \text{ if } i \in I (x)
\end{align}
for $z = x y x^{-1} \in \fS_{\tau(x)}$. 
For $x \in \,^{(l,\mu)} \fS^{(m,\nu)}$, 
put 
\begin{align*}
^{(l,\mu)} W \cap W^{(m,\nu)} (x) 
= \{ x \prod_{i \in I(x)} t_i^{a_i} \mid  a_i \in [0,r-1] \}. 
\end{align*}
We have 
\begin{align*}
^{(l,\mu)} W \cap W^{(m,\mu)} = \bigcup_{x \in \, ^{(l,\mu)} \fS^{(m,\nu)}} \, ^{(l,\mu)} W \cap W^{(m,\mu)} (x).
\end{align*} 
Thanks to \eqref{W-y I}, 
we can define an action of $\fS_{\tau(x)}$ on $^{(l,\mu)} W \cap W^{(m,\nu)} (x)$ by  
\begin{align}
\label{W act St I}
z \odot (x \prod_{i \in I (x)} t_i^{a_i}) 
=x \prod_{i \in I(x)} t_{y(i)}^{a_i} 
\end{align}
for $ z=x y x^{-1} \in \fS_{\tau(x)}$. 
We remark that, for $ z=x y x^{-1} \in \fS_{\tau(x)}$, 
we have 
\begin{align}
\label{W z act}
z \odot (x \prod_{i \in I (x)} t_i^{a_i}) 
=x \prod_{i \in I (x)} t_{y(i)}^{a_i}
= z (x \prod_{i \in I (x)} t_i^{a_i}) y^{-1}, 
\end{align}
where 
$z \in \fS_{(l,\mu)}$ and $y \in \fS_{(m,\nu)}$. 
Thus, 
for $z \in \fS_{\tau(x)}$ and $u \in ^{(l,\mu)} W \cap W^{(m,\nu)} (x)$, 
we see that 
$u$ and $z \odot u$ belong to the same $(W_{(l,\mu)}, W_{(m,\nu)})$-double coset.   

For $u  \in \,^{(l,\mu)} W \cap W^{(m,\mu)}(x)$, 
let 
$O(u) = \{ z \odot u \mid z \in \fS_{\tau(x)}\}$
be the $\fS_{\tau(x)}$-orbit under the action \eqref{W act St I}.

For $u= x \prod_{i \in I (x)} t_i^{a_i} \in \,^{(l,\mu)} W \cap W^{(m,\mu)}$,  
put 
$\ba (u) = (a_1, a_2,\dots, a_n) \in [0,r-1]^n$, 
where we put $a_i =0$ if $i \not\in I (x)$. 
Let $\succeq$ be the lexicographic order on $\ZZ^n$. 
We define a partial order $\succeq$ on $^{(l,\mu)} W \cap W^{(m,\nu)}$ by
\begin{align}
\label{W Def order succ}
u \succeq u' 
\LRa 
x= x'  \text{ and } \ba(u) \succeq \ba(u') 
\end{align}
for $u = x \prod_{i \in I(x)}t_i^{a_i}, u' = x' \prod_{i \in I(x')} t_i^{a'_i} \in \, ^{(l,\mu)} W \cap W^{(m,\nu)}$. 
Put 
\begin{align*}
^{(l,\mu)} W ^{(m, \nu)} 
= \{ u \in \, ^{(l,\mu)} W \cap W^{(m,\nu)} \mid u \text{ is minimal in } O(u) \}. 
\end{align*}
From the definition, 
any element of $^{(l,\mu)} W \cap W^{(m, \nu)}$ 
is obtained from $^{(l,\mu)} W^{(m.\nu)}$ 
by the action of $\fS_{\tau(x)}$ for $x \in {}^{(l,\mu)} \fS^{(m,\nu)}$.

\begin{lem} 
\label{W Lemma nu 1n-m}
If $\nu=(1^{n-m})$, 
we have that 
$^{(l,\mu)} W^{(m,\nu)} = \,^{(l,\mu)} W \cap W^{(m,\nu)}$. 
\end{lem}

\begin{proof} 
For any $z= x y x^{-1} \in \fS_{\tau(x)}$ ($ x \in \,^{(l,\mu)} \fS^{(m,\nu)}$), 
we have 
\begin{align*}
z \odot (x \prod_{i \in I(x)} t_i^{a_i}) = x \prod_{i \in I(x)} t_{y(i)}^{a_i} 
= x \prod_{i \in I(x)} t_i^{a_i} 
\end{align*}
since $y \in \fS_{(m,\nu)}$ and $I(x) \subset [m+1,n]$ 
together with $\nu=(1^{n-m})$. 
This implies that 
$O(u)=\{u\}$ for any $u \in \,^{(l,\mu)} W \cap W^{(m,\nu)}$, 
and we have the lemma. 
\end{proof}


For the set $^{(l,\mu)} W ^{(m,\nu)}$, we have the following proposition: 

\begin{prop}
The set 
$^{(l,\mu)} W^{(m,\nu)}$ 
is a complete set of representatives of the double cosets $W_{(l,\mu)} \backslash W_{n,r} / W_{(m,\nu)}$.
\end{prop}

\begin{proof} 
For $w =x t_1^{a_1} t_2^{a_2}\dots t_n ^{a_n} \in W_{n,r}$ ($x \in \fS_n$, $a_1, \dots, a_n \in [0,r-1]$), 
we can write 
\begin{align*} 
x= x_1 x_2 x_3 \quad (x_1 \in \fS_{(l,\mu)}, \, x_2 \in \, ^{(l,\mu)} \fS^{(m,\nu)}, \, x_3 \in \fS_{(m,\nu)}). 
\end{align*}
The relations \eqref{W-rel} imply that 
\begin{align*}
w
= x_1x_2x_3 t_1^{a_1} t_2^{a_2} \dots t_n^{a_n}
=x_1 x_2 (t_{x_3(m+1)}^{a_{m+1}} t_{x_3(m+2)}^{a_{m+2}} \dots t_{x_3 (n)}^{a_n}) x_3 t_1^{a_1} t_2^{a_2} \dots t_m^{a_m}, 
\end{align*}
where we have $\{x_3(m+1), x_3(m+2),\dots, x_3 (n)\}=[m+1, n]$ 
since $x_3 \in \fS_{(m,\nu)}$. 
Put 
$I (x_2)^c = [m+1, n] \setminus I (x_2)$. 
Then we have 
\begin{align}
\label{W w decom}
\begin{split}
w 
&= x_1 x_2 (t_{x_3(m+1)}^{a_{m+1}} t_{x_3(m+2)}^{a_{m+2}} \dots t_{x_3 (n)}^{a_n}) x_3 t_1^{a_1} t_2^{a_2} \dots t_m^{a_m} 
\\
&= x_1x_2 (\prod_{x_3(i) \in I(x_2)^c} t_{x_3(i)}^{a_i}) ( \prod_{x_3(i) \in I (x_2)} t_{x_3(i)}^{a_i}) 
	x_3 t_1^{a_1} t_2^{a_2} \dots t_m^{a_m}
\\
&= (x_1 \prod_{x_3(i) \in I (x_2)^c} t_{x_2x_3(i)}^{a_i} ) 
	( x_2 \prod_{x_3(i) \in I (x_2)} t_{x_3(i)}^{a_i})
	(x_3 t_1^{a_1} t_2^{a_2} \dots t_m^{a_m}), 
\end{split}
\end{align}
where we have 
\begin{align} 
\label{W x2 x3 i}
\{ x_2 x_3 (i) \mid x_3(i) \in I (x_2)^c \} \subset [1, l]
\end{align} 
from the definition of $I (x_2)^c$. 

Take $z=x_2 y x_2^{-1} \in \fS_{\tau (x_2)}$ such that 
$z \odot (x_2 \prod t_{x_3(i)}^{a_i})$ is minimal in $O(x_2 \prod t_{ x_3(i)}^{a_i})$, 
then \eqref{W w decom} and \eqref{W z act} imply 
\begin{align}
\label{W double coset} 
\begin{split}
w 
&= (x_1 \prod_{x_3(i) \in I (x_2)^c} t_{x_2x_3(i)}^{a_i} ) 
	( x_2 \prod_{x_3(i) \in I (x_2)} t_{x_3(i)}^{a_i})
	(x_3 t_1^{a_1} t_2^{a_2} \dots t_m^{a_m}) 
\\
&=(x_1 \prod_{x_3(i) \in I (x_2)^c} t_{x_2x_3(i)}^{a_i} ) 
	z^{-1} ( z \odot ( x_2 \prod_{x_3(i) \in I (x_2)} t_{ x_3(i)}^{a_i}) ) y
	(x_3 t_1^{a_1} t_2^{a_2} \dots t_m^{a_m})
\\
&= (x_1 z^{-1} \prod_{x_3(i) \in I (x_2)^c} t_{z x_2x_3(i)}^{a_i} ) 
	( z \odot ( x_2 \prod_{x_3(i) \in I (x_2)} t_{ x_3(i)}^{a_i}) )
	(y x_3 t_1^{a_1} t_2^{a_2} \dots t_m^{a_m}), 
\end{split}
\end{align}
where we have 
$\{z x_2 x_3 (i) \mid x_3(i) \in I (x_2)^c\} \subset [1, l]$
by \eqref{W x2 x3 i} and $z \in \fS_{(l,\mu)}$. 
From the above argument, 
we see that 
\begin{align}
\label{W-element-decom}
\begin{split}
& z \odot ( x_2 \prod_{x_3(i) \in I (x_2)} t_{x_3(i)}^{a_i})  \in \,^{(l,\mu)} W^{(m,\nu)}, 
\\
&(x_1 z^{-1} \prod_{x_3(i) \in I (x_2)^c} t_{z x_2x_3(i)}^{a_i} ) \in W_{(l,\mu)} 
\text{ and }
(y x_3 t_1^{a_1} t_2^{a_2} \dots t_m^{a_m}) \in W_{(m,\nu)}. 
\end{split}
\end{align}
The equations \eqref{W double coset} 
and \eqref{W-element-decom} imply that 
\begin{align}
\label{W-double cosets decom}
W_{n,r} 
= \bigcup_{u \in \, ^{(\l,\mu)}W^{(m,\nu)}} W_{(l,\mu)} u W_{(m,\nu)}. 
\end{align}

Finally, we prove that distinct elements of $^{(l,\mu)} W^{(m,\nu)}$ belong to 
distinct $(W_{(l,\mu)}, W_{(m,\nu)})$-double cosets. 

For $u= x \prod_{i \in I(x)} t_i^{a_i} \in \, ^{(l,\mu)} W^{(m,\nu)}$ 
and $u'=x' \prod_{i \in I (x')} t_i^{a'_i} \in \, ^{(l,\mu)} W^{(m,\nu)}$,  
suppose that $u$ and $u'$ belong to the same $(W_{(l,\mu)}, W_{(m,\nu)})$-double coset, 
namely 
$u' = w_1 u w_2$ for some 
$w_1=z \prod_{i=1}^l t_i^{b_i} \in W_{(l,\mu)}$ ($z \in \fS_{(l,\mu)}$) and 
$w_2 = y \prod_{i=1}^m t_i^{c_i} \in W_{(m,\mu)}$ ($y \in \fS_{(m,\nu)}$). 
Then we see that 
\[
x' \prod_{i \in I (x')} t_i^{a'_i} 
= ( z \prod_{i=1}^l t_i^{b_i}) ( x \prod_{i \in I(x)} t_i^{a_i}) ( y \prod_{i=1}^m t_i^{c_i} )
= z x y \prod_{i=1}^l t_{y^{-1} x^{-1} (i)}^{b_i} \prod_{i \in I(x)} t_{y^{-1}(i)}^{a_i} \prod_{i=1}^m t_i^{c_i}. 
\]
This implies that 
\begin{align}
\label{W x'=zxy}
x'= z x y \text{ and } 
\prod_{i \in I (x')} t_i^{a'_i}  
= \prod_{i=1}^l t_{y^{-1} x^{-1} (i)}^{b_i} \prod_{i \in I(x)} t_{y^{-1}(i)}^{a_i} \prod_{i=1}^m t_i^{c_i}.
\end{align}
Note that $z \in \fS_{(l,\mu)}$ and $y \in \fS_{(m,\nu)}$, 
and we see that 
$x$ and $x'$ belong to a same $(\fS_{(l,\mu)}, \fS_{(m,\nu)})$-double coset. 
Then we have $x=x'$ 
since $x, x' \in \, ^{(l,\mu)} \fS^{(m,\mu)}$. 
We also have 
$y^{-1} x^{-1} (i) = x^{-1} z (i) \not \in I(x)$ for $i \in [1,l]$
by $z \in \fS_{(l,\mu)}$ and the definition of $I(x)$. 
Thus \eqref{W x'=zxy} implies that 
\begin{align*}
x' = x = z x y, 
\quad  
\prod_{i \in I(x)} t_i^{a'_i} = \prod_{i \in I(x)} t_{y^{-1}(i)}^{a_i} 
\text{ and } 
\prod_{i=1}^l t_{y^{-1} x^{-1}(i)}^{b_i} \prod_{i=1}^m t_i^{c_i} =1, 
\end{align*}
and we have 
\begin{align*}
u' = x \prod_{i \in I(x)} t_i^{a'_i} = x \prod_{i \in I(x)} t_{y^{-1}(i)}^{a_i} 
= z \odot ( x \prod_{i \in I(x)} t_i^{a_i}) = z \odot u 
\end{align*}
since $z=x y^{-1} x^{-1} \in  \fS_{\tau(x)} = \fS_{(l,\mu)} \cap x \fS_{(m,\nu)} x^{-1} $. 
Thus we have 
$u' \in O(u)$. 
On the other hand, 
both of $u$ and $u'$ are  a minimal element in $O(u)$ since $u, u' \in \, ^{(l,\mu)} W^{(m,\nu)}$, 
and we have $u=u'$ since a minimal element in $O(u)$ is unique by the definition. 
\end{proof}

\begin{lem} 
\label{W Lemma ai}
For $u = x \prod_{i =1}^n t_i^{a_i} \in \, ^{(l,\mu)} W^{(m,\nu)}$ $(a_i =0$ if $i \not\in I(x))$ 
and $ y \in \fS_{(m,\nu)}$,
we have the followings. 
\begin{enumerate} 
\item 
$u y u^{-1} = x y x^{-1} \prod_{i=1}^n t_{x(i)}^{a_{y(i)}-a_i}$. 

\item 
$u t_j u^{-1} = t_{x(j)}$ for $j=1,2,\dots,n$. 

\item 
$a_{y(i)} = a_i =0$ if $i \in [1,m]$. 

\item 
$a_i =0$ if $x(i) \leq l$. 

\item 
$a_{y(i)} =0$ if $x(i) \leq l$ and $x y x^{-1} \in \fS_{(l,\mu)}$. 
\end{enumerate}
\end{lem}

\begin{proof} 
(\roi). 
Note that $u^{-1} = \prod_{i=1}^n t_i^{-a_i} x^{-1} = x^{-1} \prod_{i=1}^n t_{x(i)}^{- a_i}$, 
and we have 
\begin{align*}
u y u^{-1} 
= ( x \prod_{i=1}^n t_i^{a_i}) y (x^{-1} \prod_{i=1}^n t_{x(i)}^{-a_i})
= x y x^{-1} (\prod_{i=1}^n t_{x y^{-1}(i)}^{a_i}) (\prod_{i=1}^n t_{x(i)}^{-a_i}) 
= xyx^{-1} \prod_{i=1}^n t_{x(i)}^{a_{y(i)} - a_i}. 
\end{align*}

(\roii). 
For $j=1,2,\dots,n$, we have 
\begin{align*}
u t_j u^{-1} 
= (x \prod_{i=1}^n t_i^{a_i})  t_j ( \prod_{i=1}^n t_i^{- a_i} x^{-1})
= x t_j x^{-1} 
= t_{x(j)}.
\end{align*}

(\roiii). 
Note that $y \in \fS_{(m,\nu)}$ and 
$[1, m] \cap I(x)=\emptyset$, 
we have  $a_{y(i)} = a_i =0$ if $i \in [1, m]$. 

(\roiv). 
If $a_i \not=0$, 
we have $i \in I(x)$. 
Thus, we can write 
$i= x^{-1} (l +j)$ for some $j \geq 1$. 
This implies that 
$x(i) > l $ if $a_i \not=0$. 

(\rov). 
If $a_{y(i)} \not=0$, 
we can write 
$y(i) = x^{-1}(l+j)$ for some $j \geq 1$. 
This implies that 
$x(i) = x y^{-1} x^{-1} (l+j)$. 
Note that $x y^{-1} x^{-1} = (x y x^{-1})^{-1} \in \fS_{(l, \mu)}$, 
we have that 
$x(i) >l$ if $a_{y(i)} \not=0$. 
\end{proof}

\para \label{W Def cu}
For $u = x \prod_{i \in I(x)} t_i^{a_i} \in \,^{(l,\mu) }W^{(m,\nu)}$, 
put 
\begin{align*}
c(u)= \min\{ c \geq 0 \mid x(c +1) \not= c+1 {\mbox{ or }} c=n\} 
\text{ and } 
k(u) = \min\{c(u), l,m\}, 
\end{align*} 
Put 
\begin{align}
\label{W def vG(u)}
\vG(u) = ( S_{(l,\mu)} \cap \{x s_{j} x^{-1} \in x S_{(m,\nu)} x^{-1} \mid a_j = a_{j+1}\} ) 
\cup 
\{t_1,t_2,\dots, t_{k(u)}\}, 
\end{align}
where we put $a_i=0$ if $i \not\in I(x)$. 

By the definition of $I(x)$, 
we see that $1,2, \dots, k(u) \not\in I(x)$ 
since $k(u) \leq m$,  
and we have 
\begin{align}
\label{W a1=aku=0}
a_1=a_2=\dots = a_{k(u)}=0.
\end{align} 
We also see that 
\begin{align} 
\label{W sj}
s_j = x s_j x^{-1} \in S_{(l,\mu)} \cap x S_{(m,\nu)} x^{-1} 
\text{ for } j=1,2,\dots, k(u)-1 
\end{align}
since $k(u) \leq l,m $ and 
$x \in \fS_{[k(u)+1,n]}$ by Lemma \ref{S Lemma}.  

On the other hand, 
we see that 
\begin{align*}
x s_j x^{-1} (k(u))
= \begin{cases} 
	x( k(u) +1) & \text{ if } j= k(u), 
	\\
	k(u)-1 & \text{ if } j=k(u)-1, 
	\\
	k(u) \text{ otherwise} 
	\end{cases}
\end{align*}
for $j=1,2, \dots, n$ 
since $x \in \fS_{[k(u)+1,n]}$ by Lemma \ref{S Lemma}. 
This implies that 
\begin{align} 
\label{W j=k(u)}
j= k(u)  \text{ and } 
x(k(u)+1)= k(u)+1  
\text{ if }
s_{k(u)} = x s_j x^{-1}. 
\end{align} 

By \eqref{W a1=aku=0}, \eqref{W sj} and \eqref{W j=k(u)},  
we see that 
\begin{align}
\label{W vG(u)}
\{s_1,s_2,\dots, s_{k(u)-1} \} \subset \vG(u) 
\text{ and }
s_{k(u)} \not\in \vG(u), 
\end{align} 
where we note that 
$s_m \not\in S_{(m,\nu)}$, $s_l \not\in S_{(l,\mu)}$ and $x(c(u)+1) \not= c(u)+1$. 

We define a composition $\pi(u)$ of $n - k(u)$ by 
\begin{align}
\label{W Def pi(u)}
S_{(k(u), \pi(u))} = \vG(u) \cap S. 
\end{align} 
We remark that 
\begin{align}
\label{W Skupiu}
\fS_{(k(u), \pi(u))} \subset \fS_{\tau(x)}= \fS_{(l,\mu)} \cap x \fS_{(m,\nu)} x^{-1} 
\end{align}
since 
$\fS_{\tau(x)}$  
is generated by 
$S_{(l,\mu)} \cap x S_{(m,\nu)} x^{-1}$ (see \eqref{W Def taux}). 

Thanks to \eqref{W vG(u)}, 
we see that 
the subgroup of $W_{n,r}$  generated by $\vG(u)$ coincides with 
the parabolic subgroup $W_{(k(u), \pi(u))} \cong (\fS_{k(u)} \ltimes (\ZZ / r \ZZ)^{k(u)}) \times \fS_{\pi(u)}$.
We remark that 
$W_{(k(u),\pi(u))}$ is also a parabolic subgroup of $W_{(l,\mu)}$. 

\para \label{W gen twist}
For $u = x \prod_{i =1}^n t_i^{a_i} \in \,^{(l,\mu) }W^{(m,\nu)}$, 
it is clear that 
$u^{-1} W_{(k(u), \pi(u))} u$ is generated by $u^{-1} \vG(u) u$ as a subgroup of $W_{n,r}$. 
For $s_{j'}  = x s_j x^{-1} \in \vG(u) \cap S$, 
we have 
\begin{equation} 
\label{W u-1 sj u=x-1 sj x}
u^{-1} s_{j'} u  
= x^{-1} s_{j'} x \prod_{i=1}^n t_{x^{-1} (i)}^{ a_{x^{-1}(i)} - a_{x^{-1} s_{j'}(i)}} 
= s_j \prod_{i=1}^n t_{x^{-1}(i)}^{a_{x^{-1}(i)} - a_{s_j x^{-1}(i)}}  
= s_j =x^{-1} s_{j'} x,  
\end{equation}
where we note that 
$a_i = a_{s_j(i)}$ for all $i=1,2,\dots,n$ by $s_{j'}= x s_j x^{-1} \in \vG(u)$. 
On the other hand, we see that 
\begin{align}
\label{W u-1 ti u}
u^{-1} t_i u = t_i \text{ for } i \in [1, k(u)] 
\end{align}  by Lemma \ref{W Lemma ai} (\roii) and the definition of $k(u)$. 
As a consequence, we have 
\begin{align}
\label{W u-1 Gu u}
u^{-1} \vG(u) u 
= (x^{-1} S_{(l,\mu)} x \cap \{s_j \in S_{(m,\nu)} \mid a_j = a_{j+1}\} )
	\cup \{t _1, t_2, \dots, t_{k(u)}\}. 
\end{align}
Moreover, we see that 
\begin{align}
\label{W j'=k(u)}
j' = k(u) \text{ and } x^{-1} (k(u) +1) = k(u) +1 \text{ if } s_{k(u)} = x^{-1} s_{j'} x. 
\end{align}
in a similar way to \eqref{W j=k(u)}. 
By \eqref{W a1=aku=0}, \eqref{W sj} and \eqref{W j'=k(u)},  
we see that 
\begin{align}
\label{W u-1 vG(u) u}
\{s_1,s_2,\dots, s_{k(u)-1} \} \subset u^{-1} \vG(u) u
\text{ and }
s_{k(u)} \not\in u^{-1} \vG(u) u, 
\end{align} 
where we note that 
$s_m \not\in S_{(m,\nu)}$, $s_l \not\in S_{(l,\mu)}$ and $x^{-1} (c(u)+1) \not= c(u)+1$. 

We define a composition $\pi^{\sharp}(u)$ of $n- k(u)$ by 
$S_{(k(u), \pi^{\sharp}(u))} = u^{-1} \vG(u) u \cap S$. 
Then we see that 
\begin{align}
\label{W S ku pi sharp u}
S_{(k(u), \pi^{\sharp}(u))} = x^{-1} S_{(k(u), \pi(u))} x 
\end{align}
and 
\begin{align}
\label{W u-1 Wkupiu u}
u^{-1} W_{(k(u), \pi(u))} u= W_{(k(u), \pi^{\sharp}(u))} \cong (\fS_{(k(u)} \ltimes (\ZZ/ r \ZZ)^{k(u)}) \times \fS_{\pi^{\sharp}(u)} 
\end{align}
by \eqref{W u-1 Gu u} and \eqref{W u-1 vG(u) u}. 
In particular, 
$u^{-1} W_{(k(u), \pi(u))} u$ is a parabolic subgroup of $W_{(m,\nu)}$.  


\begin{prop}
\label{W prop only braid}
For $ u = x \prod_{i \in I(X)} t_i^{a_i} \in {}^{(l,\mu)} W^{(m,\nu)}$, 
we have $W_{(k(u), \pi(u))} = u W_{(k(u), \pi^{\sharp}(u))} u^{-1}$ 
and $X_{(k(u), \pi(u))} = u X_{(k(u), \pi^{\sharp}(u))} u^{-1}$. 
In particular, 
for $s_j \in X_{(k(u), \pi(u))}$, 
there exists $s_{\psi(j)} \in X_{(k(u), \pi^{\sharp}(u))}$ such that 
$s_j = u s_{\psi(j)} u^{-1}$.  

Moreover, 
the identity  
\[
s_j  (s_{i_1} s_{i_2} \dots s_{i_l} \prod_{i \in I(X)} t_i^{a_i}) = (s_{i_1} s_{i_2} \dots s_{i_l} \prod_{i \in I(X)} t_i^{a_i}) s_{\psi(j)} 
\text{ for } s_j \in X_{(k(u), \pi(u))} 
\]
follows only from the braid relations associated with $W_{n,r}$, 
where $x = s_{i_1} s_{i_2} \dots s_{i_l}$ is a reduced expression of $x \in \fS_n$. 

\begin{proof} 
For $u = x \prod_{i \in I(X)} t_i^{a_i} \in {}^{(l,\mu)}W^{(m,\nu)}$, 
we have already seen that 
$W_{(k(u), \pi(u))} = u W_{(k(u), \pi^{\sharp}(u))} u^{-1}$ 
and 
$X_{(k(u), \pi(u))} = u X_{(k(u), \pi^{\sharp}(u))} u^{-1}$. 
Thus, for $s_j \in X_{(k(u), \pi(u))}$, 
there exists $s_{\psi(j)} \in X_{(k(u), \pi^{\sharp}(u))}$ such that 
$s_j = u s_{\psi(j)} u^{-1}$. 
Let $x = s_{i_1} s_{i_2} \dots s_{i_l}$ be a reduced expression of $x \in \fS_n$. 

It is easy to check that
the relations 
\begin{align}
\label{W rel follows braid relation}
\begin{split}
&s_i s_j = s_j s_i  \text{ if } |i-j|>1, 
\\
&t_i t_j = t_j t_i \,  (1\leq i , j \leq n), 
\\ 
&s_i t_j = t_j s_i \text{ if } j \not=i, i+1 
\\
&s_i t_i t_{i+1} = t_i t_{i+1} s_i 
\end{split}
\end{align}
follow only from the braid relations associated with $W_{n,r}$ by direct calculation.

Note that $x \in \fS_{[k(u)+1, n]}$ by Lemma \ref{S Lemma} and $s_0 = t_1$, 
we have 
$s_0 (s_{i_1} s_{i_2} \dots s_{i_l} \prod_{i \in I(x)} t_i^{a_i}) = (s_{i_1} s_{i_2} \dots s_{i_l} \prod_{i \in I(x)} t_i^{a_i}) s_0$ if $k(u) \not=0$ 
and this identity follows only from the braid relations. 

For $s_j \in X_{(k(u),\pi(u))} \setminus \{s_0\}$, 
we see that 
$s_j = x s_{\psi(j)} x^{-1}$ and $a_{\psi(j)} = a_{\psi(j)+1}$ 
by \eqref{W u-1 sj u=x-1 sj x}. 
Moreover, 
we see that 
$\ell (s_j x) = \ell(x) +1 = \ell(x s_{\psi(j)})$ 
since $s_j \in  \fS_{(l,\mu)}$, $s_{\psi(j)} \in \fS_{(m,\nu)}$ 
(see \eqref{W def vG(u)} and \eqref{W u-1 Gu u}),
and 
$x \in {}^{(l,\mu)} \fS^{(m,\nu)} = {}^{(l,\mu)} \fS \cap \fS^{(m,\nu)}$. 
Thus, 
the identity  
$s_j (s_{i_1} s_{i_2} \dots s_{i_l}) = (s_{i_1} s_{i_2} \dots s_{i_l}) s_{\psi(j)}$ follows only from the braid relations associated with $\fS_n$ 
by the general theory of Coxeter groups. 
Then we see that 
the identity $s_j  (s_{i_1} s_{i_2} \dots s_{i_l} \prod_{i \in I(X)} t_i^{a_i}) = (s_{i_1} s_{i_2} \dots s_{i_l} \prod_{i \in I(X)} t_i^{a_i}) s_{\psi(j)}$ 
follows only from the braid relations 
(note that $a_{\psi(j)} = a_{\psi(j)+1}$). 
\end{proof}

\end{prop}


\begin{prop}
\label{W Prop Wkupiu}
For $u = x \prod_{i \in I(x)} t_i^{a_i} \in \, ^{(l,\mu)} W^{(m,\nu)}$, 
the subgroup $W_{(l,\mu)} \cap u W_{(m,\nu)} u^{-1}$ of $W_{n,r}$ 
is generated by $\vG(u)$. 
In particular, 
we have 
\[
W_{(l,\mu)} \cap u W_{(m,\nu)} u^{-1}  
= W_{(k(u), \pi(u))} \cong (\fS_{k(u)} \ltimes (\ZZ / r \ZZ)^{k(u)}) \times \fS_{\tau(u)}. 
\]
\end{prop}

\begin{proof} 
Put $a_j =0$ for $j \not\in I(x)$. 
For $w=y \prod_{i=1}^m t_i^{b_i} \in W_{(m,\nu)}$ ($y \in \fS_{(m,\nu)}$), 
we have 
\begin{align*}
u w u^{-1} 
= uyu^{-1} \prod_{i=1}^m (u t_i u^{-1})^{b_i}
= x y x^{-1} (\prod_{i=1}^n t_{x(i)}^{a_{y(i)} - a_i}) (\prod_{i=1}^m t_{x(i)}^{b_i}) 
\end{align*} 
by Lemma \ref{W Lemma ai} (\roi) and (\roii). 
We have 
\begin{align}
\label{W ayi 1 to m}
a_{y(i)} = a_i =0  \text{ for } i \in [1,m] 
\end{align}  
by Lemma \ref{W Lemma ai} (\roiii). 
Thus we have 
\begin{align}
\label{W uwu-1}
u w u^{-1}  
= x y x^{-1} (\prod_{i=1}^m t_{x(i)}^{b_i}) (\prod_{i=m+1}^n t_{x(i)}^{a_{y(i)} - a_i}).
\end{align}
Suppose that $u w u^{-1} \in W_{(l,\mu)} = (\fS_{l} \ltimes (\ZZ/r\ZZ)^l) \times \fS_{\mu}^{[l]}$. 
Then we have 
\begin{align}
\label{W ayi bi}
xyx^{-1} \in \fS_{(l,\mu)}, 
\quad 
a_{y(i)} = a_i \text{ if } x(i) >l, 
\quad 
b_i =0 \text{ if } x(i) >l 
\end{align}
by \eqref{W uwu-1}. 
Since $x y x^{-1} \in \fS_{(l,\mu)}$, 
we have 
\begin{align}
\label{W ayi=0}
a_{y(i)} = a_i =0 \text{ if } x(i) \leq l 
\end{align}
by Lemma \ref{W Lemma ai} (\roiv) and (\rov). 
Moreover, we see that 
\begin{align}
\label{W 1 to ku}
[1, l] \cap \{x(1), \dots, x(m)\} =[1, k(u)]
\end{align} 
by Lemma \ref{S Lemma}, 
where we note that 
$x(i) =i$ for $i \in [1, k(u)]$. 
As a consequence 
of \eqref{W ayi 1 to m}, \eqref{W uwu-1}, \eqref{W ayi bi}, \eqref{W ayi=0} and \eqref{W 1 to ku},
we have that 
\begin{align}
\label{W xyx-1}
x y x^{-1} \in \fS_{(l,\mu)}, \, 
a_{y(i)} = a_i \, (1\leq i \leq n) 
\text{ and } 
b_j =0 \, (j > k(u)) 
\end{align}
if $uwu^{-1} \in W_{(l,\mu)}$. 
On the other hand, 
it is clear that $u w u^{-1} \in W_{(l,\mu)}$ if \eqref{W xyx-1} holds 
for $w=y \prod_{i=1}^m t_i^{b_i} \in W_{(m,\nu)}$ ($y \in \fS_{(m,\nu)}$). 
Thus, we see that 
$W_{(l,\mu)} \cap u W_{(m,\nu)} u^{-1}$ is generated by 
\begin{align*}
\hat{\vG}(u) := 
&(\fS_{(l,\mu)} \cap \{ x y x^{-1} \mid y\in \fS_{(m,\nu)} \text{ such that } a_{y(i)} =a_i \, (1\leq i \leq n)\} ) 
\\
&\cup 
\{ t_1, t_2, \dots, t_{k(u)}\}. 
\end{align*}

For $z = x y x^{-1} \in \fS_{\tau(x)} =\fS_{(l,\mu)} \cap x \fS_{(m,\nu)} x^{-1}$, 
let $y=s_{i_1} s_{i_2} \dots s_{i_p}$ be a reduced expression. 
Then we see that 
$x s_{i_j} x^{-1} \in  \fS_{\tau(x)}$ ($j=1,2,\dots,p$)
since 
$\fS_{\tau(x)}$  
is generated by 
$S_{(l,\mu)} \cap x S_{(m,\nu)} x^{-1}$. 
We claim that 
\begin{align}
\label{W claim a}
a_{i_j} = a_{i_{j+1}} \, (1 \leq j \leq p) 
\text{ if } 
a_{y(i)} = a_i \, (1\leq i \leq n).
\end{align}
Then the claim \eqref{W claim a} implies that 
$\hat{\vG}(u) \supset \vG(u)$, 
and we easily see that 
$W_{(l,\mu)} \cap u W_{(m,\nu)} u^{-1}$ 
is generated by 
$\vG(u)$.

We prove the claim \eqref{W claim a}. 
We have 
\begin{align}
\label{W i}
\{i_1,i_2,\dots, i_k\} = \bigcup_{i< y(i)} \{i, i+1, \dots,  y(i)-1\} \cup \bigcup_{i> y(i)} \{y(i), y(i)+1,\dots, i-1\}.
\end{align}
Suppose that $i < y(i)$ and $a_{y(i)} =a_i$. 
By \eqref{W i}, 
we see that $x s_j x^{-1} \in \fS_{\tau(x)}$ ($i \leq j < y(i)$), 
and we have 
\begin{align*}
(x s_j x^{-1}) \odot u 
= x (t_1^{a_1} \dots t_{j-1}^{a_{j-1}}) (t_j^{a_{j+1}} t_{j+1}^{a_j}) (t_{j+2}^{a_{j+2}} \dots t_n^{a_n}) 
\in O(u).
\end{align*}
If there exists $j$ ($i \leq j <y(i)$) such that $a_i \leq a_{i+1} \leq  \dots \leq a_{j}$ and $a_j > a_{j+1}$, 
we have 
$\ba(u) \succ \ba((x s_j x^{-1}) \odot u)$. 
This is a contradiction since $u$ is minimal  in $O(u)$. 
Thus, we have $a_i \leq a_{i+1} \leq \dots \leq a_{y(i)}$, 
and $a_i = a_{i+1} = \dots = a_{y(i)}$ by $a_{y(i)}=a_i$. 
Similarly, we have 
$a_{y(i)} = a_{y(i)+1} = \dots = a_i$ if $i > y(i)$ and $a_{y(i)} =a_i$. 
Then we have the claim \eqref{W claim a}. 
\end{proof}


\para \label{W Def rep Wlmu Wkupiu}
For $u \in \,^{(l,\mu)} W^{(m,\nu)}$, 
the group $W_{(k(u), \pi(u))} = W_{(l,\mu)} \cap u W_{(m,\nu)} u^{-1}$ 
is a parabolic subgroup of $W_{(l,\mu)}$ 
by  Proposition \ref{W Prop Wkupiu}. 
Put 
\begin{multline*}
(W_{(l,\mu)})^{(k(u), \pi(u))} 
\\
= \{ x t_{k(u)+1}^{a_{k(u)+1}} t_{k(u)+2}^{a_{k(u)+2}} \dots t_l^{a_l} 
	\mid x \in (\fS_{(l,\mu)})^{(k(u), \pi(u))}, \, 
	a_{k(u)+1}, \dots, a_l \in [0,r-1] \}, 
\end{multline*}
where 
$(\fS_{(l,\mu)})^{(k(u), \pi(u))}$ is the set of distinguished coset representatives of 
the cosets 
$\fS_{(l,\mu)} / \fS_{(k(u), \pi(u))}$. 
Then $(W_{(l,\mu)})^{(k(u), \pi(u))} $ is a complete set of representatives of $W_{(l,\mu)}/ W_{(k(u), \pi(u))}$ 
which is proven in a similar way to the proof of Lemma \ref{W Lemma coset rep}. 
We have the following corollary. 


\begin{cor}
\label{W Cor}
For each $u \in \,^{(l,\mu)} W^{(m,\nu)}$, 
the multiplication map $($in $W)$ 
\begin{align*}
(W_{(l,\mu)})^{(k(u), \pi(u))} \times \{u\} \times W_{(m,\nu)} 
\ra W_{(l,\mu)} u W_{(m,\nu)}, 
\quad 
(w_1, u, w_2) \mapsto w_1 u w_2
\end{align*} 
is a bijection. 
\end{cor}

\begin{proof} 
By definitions, 
it is clear that the map is surjective. 
On the other hand, 
if $w_1 u w_2 = w'_1 u w'_2$ 
for 
$w_1, w'_1 \in (W_{(l,\mu)})^{(k(u), \pi(u))}$ 
and 
$w_2, w'_2 \in W_{(m,\nu)}$, 
we have 
\begin{align*} 
w_1^{-1} w'_1 = u w_2 w'^{-1}_2 u^{-1} \in W_{(l,\mu)} \cap u W_{(m,\nu)} u^{-1} = W_{(k(u), \pi(u))}.
\end{align*}
This implies that $w_1 = w'_1$ 
since $w_1, w'_1 \in (W_{(l,\mu)})^{(k(u), \pi(u))}$. 
Thus we also have $w_2 = w'_2$,  
and the map is injective. 
\end{proof}



\section{The Mackey formula for cyclotomic Hecke algebras}
\label{H Section}

In this section, 
we construct various $R$-free basis of the cyclotomic Hecke algebra $\sH_{n,r}$ 
associated with $W_{n,r}$ 
which are compatible with the decomposition of $W_{n,r}$ to the cosets 
$W/ W_{(l,\mu)}$ and the double cosets $W_{(l,\mu)} \backslash W_{n,r} / W_{(m,\nu)}$. 
Then we establish the Mackey formula for cyclotomic Hecke algebras. 

\para \label{H Def}
Let $R$ be a commutative ring, 
and take parameters 
$q, Q_1, Q_2,\dots, Q_r \in R$ such that $q$ is invertible in $R$. 
The cyclotomic Hecke algebra (Ariki-Koike algebra) $\sH_{n,r} =\sH (W_{n,r})$ 
associated with $W_{n,r}$ is the associative algebra with $1$ over $R$ 
generated by 
$T_0, T_1, \dots, T_{n-1}$ with the following defining relations: 
\begin{align}
\label{H relations}
\begin{split} 
&(T_0-Q_1)(T_0-Q_2) \dots (T_0-Q_r)=0, 
\quad 
(T_i +1)(T_i-q) =0 \quad (1 \leq i \leq n-1), 
\\
& T_0 T_1 T_0 T_1 = T_1 T_0 T_1 T_0, 
\quad T_i T_{i+1} T_i = T_{i+1} T_i T_{i+1} \quad (1 \leq i \leq n-2), 
\\
& T_i T_j = T_j T_i \quad (|i-j|>1). 
\end{split}
\end{align}

The subalgebra of $\sH_{n,r}$ generated by 
$T_1, T_2,\dots, T_{n-1}$ is isomorphic to the Iwahori-Hecke algebra $\sH(\fS_n)$ associated with $\fS_n$. 
For $x \in \fS_n$, 
put $T_x = T_{i_1} T_{i_2} \dots T_{i_l}$ for a reduced expression $x=s_{i_1} s_{i_2} \dots s_{i_l}$, 
and $\{T_x \mid x \in \fS_n\}$ is an $R$-free basis of $\sH(\fS_n)$. 

Set $L_i= q^{1-i} T_{i-1} \dots T_1 T_0 T_1 \dots T_{i-1}$ for $i=1,2, \dots, n$.  
For $w = x t_1^{a_1} \dots t_n^{a_n} \in W_{n,r}$ ($x \in \fS_n$, $ a_1,\dots, a_n \in [0,r-1]$), 
put 
$T_w = T_x L_1^{a_1} L_2^{a_2} \dots L_n^{a_n}$. 
Then we have that 
$\{T_w \mid w \in W_{n,r}\}$ 
is an $R$-free basis of $\sH_{n,r}$ by \cite[Theorem 3.10]{AK}. 

For a parabolic subgroup $W_{(l,\mu)}$ of $W_{n,r}$, 
we define the subalgebra $\sH_{(l,\mu)}$ 
of $\sH_{n,r}$ generated by 
$T_0$ (in the case where $l \geq 1$) and $T_x$ ($x \in \fS_{(l,\mu)}$) 
is isomorphic to the cyclotomic Hecke algebra $\sH(W_{(l,\mu)})$ 
associated with $W_{(l,\mu)}$. 
We see easily that 
$\{T_w \mid w \in W_{(l,\mu)}\}$ 
is an $R$-free basis of $\sH_{(l,\mu)}$. 

The following properties are well known, 
and one can check them by direct calculation using the defining relations. 

\begin{lem}
\label{H Lemma L T}
We have the following. 
\begin{enumerate} 
\item 
$L_i$ and $L_j$ commute with each other for any $1\leq i, j \leq n$. 

\item 
$T_i$ and $L_j$ commute with each other if $j \not=i, i+1$. 

\item 
$T_i$ commutes with both $L_i L_{i+1}$ and $L_i + L_{i+1}$. 

\item 
$L_{i+1}^b T_i = T_i L_i^b + (q-1) \sum_{c=0}^{b-1} L_i^c L_{i+1}^{b-c}$.  

\item 
$L_i^b T_i = T_i L_{i+1}^b - (q-1) \sum_{c=0}^{b-1} L_i^c L_{i+1}^{b-c}$. 

\end{enumerate}
\end{lem}

Lemma \ref{H Lemma L T} implies the following lemma: 
\begin{lem} 
\label{H Lemma Tx L}
 For $k \geq 0$, $x \in \fS_{(k, n-k)}  $ and  
 $ a_{k+1},\dots, a_n \in {[0,r-1]} $, 
we have 
\begin{align*}
&T_x (L_{k+1}^{a_{k+1}} L_{k+2}^{a_{k+2}} \dots L_{n}^{a_n})
\\
&=(L_{x(k+1)}^{a_{k+1}} L_{x(k+2)}^{a_{k+2}} \dots L_{x(n)}^{a_n})T_x 
 + \sum_{y  < x} \sum_{(b_{k+1},\dots, b_{n}) \in
 {[0,r-1]}^{n-k} }
	r_{y}^{(b_{k+1},\dots, b_n)} T_y (L_{k+1}^{b_{k+1}} L_{k+2}^{b_{k+2}} \dots L_n^{b_n}) 
\end{align*}
for some $r_{y}^{(b_1,\dots, b_n)} \in R$. 
\end{lem}

\begin{proof} 
We prove the lemma by the induction on $\ell(x)$. 
If $\ell(x)=0$, it is clear. 
Suppose that $\ell(x) >0$. 
Let $x = s_{i_1} s_{i_2} \dots s_{i_l}$ be a reduced expression, 
and put $x'= x s_{i_l}$. 
We have $T_x = T_{x'} T_{i_l}$, and we see that 
\begin{align*}
&T_x (L_{k+1}^{a_{k+1}} L_{k+2}^{a_{k+2}} \dots L_n^{a_n}) 
\\
&= \begin{cases} 
	T_{x'} (L_{s_{i_l}(k+1)}^{a_{k+1}} L_{s_{i_l}(k+2)}^{a_{k+2}} \dots L_{s_{i_l}(n)}^{a_n}) T_{i_l} 
	 & \text{ if } a_{i_l} = a_{i_l+1}, 
	\\
	T_{x'} (L_{s_{i_l}(k+1)}^{a_{k+1}} L_{s_{i_l}(k+2)}^{a_{k+2}} \dots L_{s_{i_l}(n)}^{a_n}) T_{i_l} 
		\\ \quad 
		+ (q-1) \sum_{c=a_{i_l}}^{a_{i_l+1}-1} T_{x'} 
		(L_{k+1}^{a_{k+1}} \dots L_{i_l -1}^{a_{i_l -1}}) (L_{i_l}^{c} L_{i_l+1}^{a_{i_l} + a_{i_{l+1}}-c}) 
		(L_{i_l+2}^{a_{i_l+2}} \dots L_n^{a_n})
		\hspace{-5em}
		\\ 
		& \text{ if } a_{i_l} < a_{i_l +1}, 
	\\
	T_{x'} (L_{s_{i_l}(k+1)}^{a_{k+1}} L_{s_{i_l}(k+2)}^{a_{k+2}} \dots L_{s_{i_l}(n)}^{a_n}) T_{i_l} 
		\\ \quad 
		- (q-1) \sum_{c=a_{i_l+1}}^{a_{i_l}-1} T_{x'} 
		(L_{k+1}^{a_{k+1}} \dots L_{i_l -1}^{a_{i_l -1}}) (L_{i_l}^{c} L_{i_l+1}^{a_{i_l} + a_{i_{l+1}}-c}) 
		(L_{i_l+2}^{a_{i_l+2}} \dots L_n^{a_n})
		\hspace{-5em}
		\\ 
		& \text{ if } a_{i_l} > a_{i_l +1} 
	\end{cases} 
\end{align*}
by direct calculation using Lemma \ref{H Lemma L T}. 
Applying the assumption of the induction to 
$T_{x'} (L_{s_{i_l}(k+1)}^{a_{k+1}} L_{s_{i_l}(k+2)}^{a_{k+2}} \dots L_{s_{i_l}(n)}^{a_n})$, 
we have the lemma. 
\end{proof}

\begin{prop}
\label{H Prop Hnr Hlmu}
For $W_{(l,\mu)}$, a parabolic subgroup of $W_{n,r}$, 
the elements
\[
 \{T_{w_1} T_{w_2} \mid w_1 \in W^{(l,\mu)}, \, w_2 \in W_{(l,\mu)}\}
\]
is an $R$-free basis of $\sH_{n,r}$. 
Moreover 
$\sH_{n,r}$ 
is a free right $\sH_{(l,\mu)}$-module with an $\sH_{(l,\mu)}$-free basis 
$\{T_w \mid w \in W^{(l,\mu)} \}$. 
\end{prop}

\begin{proof} 
For $w = x t_1^{a_1} \dots t_n^{a_n} \in W_{n,r}$ ($x \in \fS_n$), 
we can write 
$x = x_1 x_2$ ($x_1 \in \fS^{(l,\mu)}$, $x_2 \in \fS_{(l,\mu)}$), 
and 
$x_2 = y_1 y_2$ ($y_1 \in \fS_l$, $y_2 \in \fS_\mu^{[l]}$). 
Note that $\ell(x) = \ell(x_1) + \ell(x_2)$ 
and $\ell(x_2) = \ell (y_1) + \ell (y_2)$, 
we have 
\begin{align*}
T_w 
&= T_x L_1^{a_1} \dots L_n^{a_n} 
\\
&= T_{x_1} T_{x_2} L_1^{a_1} \dots L_n^{a_n} 
\\
&= T_{x_1} T_{y_1} T_{y_2} L_1^{a_1} \dots L_n^{a_n} 
\\
&= T_{x_1} T_{y_2} (L_{l+1}^{a_{l+1}} L_{l+2}^{a_{l+2}} \dots L_n^{a_n}) T_{y_1} (L_1^{a_1} L_2^{a_2} \dots L_l^{a_l}),
\end{align*}
where we use Lemma \ref{H Lemma L T} (\roi) and (\roii) in the last equation. 
Note that $y_2 \in \fS_{\mu}^{[l]}$, 
we see that 
\begin{align*}
&T_{y_2} (L_{l+1}^{a_{l+1}} L_{l+2}^{a_{l+2}} \dots L_n^{a_n})
\\
&= L_{y_2(l+1)}^{a_{l+1}} L_{y_2(l+2)}^{a_{l+2}} \dots L_{y_2 (n)}^{a_n} T_{y_2} 
 + \sum_{z < y_2 } \sum_{(b_{l+1}, \dots, b_n) \in
{[0,r-1]}^{n-l}}
 r_z^{(b_{l+1}, \dots, b_n)}   
	L_{l+1}^{b_{l+1}} \dots L_n^{b_n} T_z
\end{align*}
by using Lemma \ref{H Lemma Tx L} repeatedly. 
Thus we have 
\begin{align}
\label{W Tw coset decom}
\begin{split}
T_w 
&= T_{x_1} L_{y_2(l+1)}^{a_{l+1}} L_{y_2(l+2)}^{a_{l+2}} \dots L_{y_2 (n)}^{a_n} 
	T_{y_2} T_{y_1} (L_1^{a_1} L_2^{a_2} \dots L_l^{a_l})
	\\ & \quad 
 +\sum_{z < y_2 } \sum_{(b_{l+1}, \dots, b_n) \in
 {[0,r-1]}^{n-l}}
		r_z^{(b_{l+1}, \dots, b_n)} 
	T_{x_1} L_{l+1}^{b_{l+1}} \dots L_n^{b_n} T_z T_{y_1} (L_1^{a_1} L_2^{a_2} \dots L_l^{a_l}) 
\\
&=  (T_{x_1} L_{y_2(l+1)}^{a_{l+1}} L_{y_2(l+2)}^{a_{l+2}} \dots L_{y_2 (n)}^{a_n})
	(T_{x_2}  L_1^{a_1} L_2^{a_2} \dots L_l^{a_l})
	\\ & \quad 
 +\sum_{z < y_2 } \sum_{(b_{l+1}, \dots, b_n) \in
 {[0,r-1]}^{n-l}}
	 r_z^{(b_{l+1}, \dots, b_n)} 
	(T_{x_1} L_{l+1}^{b_{l+1}} \dots L_n^{b_n}) (T_{y_1 z} L_1^{a_1} L_2^{a_2} \dots L_l^{a_l}), 
\end{split}
\end{align}
where we note that $y_1 z  < x_2 = y_1 y_2$.  

We define a partial order $\geq$ on $W_{(l,\mu)}$ by 
$w = x t_1^{a_1} \dots t_l^{a_l} \geq w' = x' t_1^{a'_1} \dots t_l^{a'_l}$ 
if $x > x'$.   
Then we have 
\begin{align*}
T_w = T_{w_1} T_{w_2} + \sum_{w'_1 \in W^{(l,\mu)}, w'_2 \in W_{(l,\mu)} \atop w'_2 < w_2} 
	r_{w'_1, w'_2} T_{w'_1} T_{w'_2}
\end{align*}
by the equations \eqref{W Tw coset decom}, 
where $w_1 = x_1 t_{y_2(l+1)}^{a_l+1} t_{y_2(l+2)}^{a_l+2} \dots t_{y_2(n)}^{a_n}$ 
and $w_2 = x_2 t_1^{a_1} \dots t_l^{a_l}$. 

This implies that 
$\{ T_{w_1} T_{w_2} \mid w_1 \in W^{(l,\mu)}, \, w_2 \in W_{(l,\mu)}\}$ 
is an $R$-free basis of $\sH_{n,r}$, 
and we have $\sH_{n,r} = \bigoplus_{w \in W^{(l,\mu)}} T_w \sH_{(l,\mu)}$ 
as right $\sH_{(l,\mu)}$-modules.  
\end{proof}


\para 
Recall that 
$W_{(k(u), \pi(u))} = W_{(l,\mu)} \cap u W_{(m,\nu)} u^{-1}$ 
for $u = x \prod_{i \in I(x)} t_i^{a_i} \in \,^{(l,\mu)}W^{(m,\nu)}$, 
and we have 
\begin{align*}
W_{(k(u), \pi(u))} = \{ z t_1^{a_1} t_2^{a_2} \dots t_{k(u)}^{a_{k(u)}} \mid z \in \fS_{(k(u),\pi(u))}, \, 
a_1,\dots, a_{k(u)} \in [0,r-1] \}. 
\end{align*}
Then we see that 
$\sH_{(k(u), \pi(u))}$
has an $R$-free basis 
\begin{align}
\label{H basis Hkupiu}
\{T_z L_1^{a_1} L_2^{a_2} \dots L_{k(u)}^{a_{k(u)}} \mid 
z \in \fS_{(k(u),\pi(u))}, \,  a_1,\dots, a_{k(u)} \in [0,r-1] \}.
\end{align}


\begin{prop}
\label{H Prop TuHmnu}
For each $u = x \prod_{i \in I(x)} t_i^{a_i} \in \,^{(l,\mu)}W^{(m,\nu)}$, 
we have the following: 
\begin{enumerate} 
\item 
$L_i T_u =T_u L_i$ for $i=1,2,\dots, k(u)$.  

\item 
$T_z  T_u= T_u T_{ x^{-1} z x}$ for $z \in \fS_{(k(u), \pi(u))}$. 
\end{enumerate}
In particular, 
$T_u \sH_{(m,\nu)}$ has an 
$(\sH_{(k(u), \pi(u))}, \sH_{(m,\nu)})$-bimodule structure by multiplications in $\sH_{n,r}$. 
More precisely, for $T_u Y \in T_u \sH_{(m,\nu)}$, we have 
\begin{align*}
L_i (T_u Y) = T_u (L_i Y) \quad (1 \leq i \leq k(u)), 
\quad 
T_z (T_u Y) = T_u (T_{x^{-1} z x} Y) \quad (z \in \fS_{(k(u), \pi(u))}). 
\end{align*}
\end{prop}

\begin{proof} 
Note the definition of the element $T_w \in \sH_{n,r}$ for $w \in W_{n,r}$, 
this proposition follows from Proposition \ref{W prop only braid} 
together with \eqref{W u-1 sj u=x-1 sj x} and \eqref{W u-1 ti u}.
\end{proof}


\para 
For $u = x \prod_{i \in I(x)} t_i^{a_i} \in \,^{(l,\mu)} W^{(m,\nu)}$, 
recall that $W_{(k(u), \pi(u))} = W_{(l,\mu)} \cap u W_{(m,\nu)} u^{-1}$ and  
$W_{(k(u), \pi^{\sharp}(u))} = u^{-1} W_{(k(u), \pi(u))} u$ (see \eqref{W u-1 Wkupiu u}) 
are parabolic subgroup of $W_{n,r}$. 
Then we see that 
$\sH_{(k(u), \pi(u))}$ (resp. $\sH_{(k(u), \pi^{\sharp}(u))}$) 
has an $R$-free basis 
\begin{align*} 
&\{T_z L_1^{a_1} \dots L_{k(u)}^{a_{(k(u)}} \mid z \in \fS_{(k(u), \pi(u))}, \, a_1, \dots a_{k(u)} \in [0, z-1]\}
\\
&(\text{resp. } 
	\{T_{y} L_1^{a_1} \dots L_{k(u)}^{a_{k(u)}} \mid y \in \fS_{(k(u), \pi^{\sharp}(u))}, \, a_1, \dots a_{k(u)} \in [0, z-1]\}), 
\end{align*}
where we note that $\fS_{(k(u), \pi^{\sharp}(u))} = x^{-1} \fS_{(k(u), \pi(u))} x$ by \eqref{W S ku pi sharp u}. 
We have the following corollary. 

\begin{cor} 
\label{H Cor Tu func}
For $u = x \prod_{i \in I(x)} t_i^{a_i} \in \,^{(l,\mu)} W^{(m,\nu)}$, we have the following: 
\begin{enumerate} 
\item 
$T_u \sH_{(k(u), \pi^{\sharp}(u))}$ has an $(\sH_{(k(u), \pi(u))}, \sH_{(k(u), \pi^{\sharp}(u))})$-bimodule structure 
by multiplications in $\sH_{n,r}$. 

\item 
We have the isomorphism of $(\sH_{(k(u), \pi(u))}, \sH_{(m,\nu)})$-bimodules 
\begin{align*}
T_u \sH_{(m,\nu)} \cong T_u \sH_{(k(u), \pi^{\sharp}(u))} 
\otimes_{\sH_{(k(u), \pi^{\sharp}(u))} } \sH_{(m,\nu)}.
\end{align*}

\end{enumerate}
\end{cor}
 
\begin{proof} 
(\roi) follows from Proposition \ref{H Prop TuHmnu} 
(note that $\fS_{(k(u), \pi^{\sharp}(u))} = x^{-1} \fS_{(k(u), \pi(u))} x$). 
Note that $ u \in \,^{(l, \mu)} W ^{(m,\nu)} \subset W^{(m,\nu)}$, 
and $\sH_{(k(u), \pi^{\sharp}(u))}$ is a subalgebra of $\sH_{(m,\nu)}$.
Then,  
by Proposition \ref{H Prop Hnr Hlmu}, 
we see that 
$T_u \sH_{(k(u), \pi^{\sharp}(u))} \cong \sH_{(k(u), \pi^{\sharp}(u))}$ 
as right $\sH_{(k(u), \pi^{\sharp}(u))}$-modules, 
and we have (\roii). 
\end{proof}

\para 
By Corollary \ref{W Cor}, 
any element $w \in W_{n,r}$ is uniquely written as 
\begin{align}
w = w_1 u w_2  
\quad (u \in \, ^{(l,\mu)} W^{(m,\nu)}, \, 
w_1 \in (W_{(l,\mu)})^{(k(u), \pi(u))}, \, 
w_2 \in W_{(m.\nu)}). 
\end{align}
By using this decomposition, 
we define $\wt{T}_w \in \sH_{n,r}$ by 
$\wt{T}_w = T_{w_1} T_u T_{w_2}$. 
Then we have the following proposition. 

\begin{prop}  
\label{H H decom}
We have the following. 
\begin{enumerate} 
\item 
$\displaystyle \sH_{n,r} = \sum_{u \in \,^{(l,\mu)} W^{(m,\nu)}} \sH_{(l,\mu)} T_u \sH_{(m,\nu)}$. 

\item 
$\{\wt{T}_w \mid w \in W_{n,r}\}$ is an $R$-free basis of $\sH_{n,r}$. 
\end{enumerate} 
\end{prop}

\begin{proof} 
We prove (\roi). 
First we prove that 
\begin{align}
\label{H Tv}
T_v \in \sum_{ u \in \,^{(l,\mu)} W^{(m,\nu)}} \sH_{(l,\mu)} T_u \sH_{(m,\nu)}  
\text{ for any } v =x \prod_{i \in I(x)} t_i^{a_i}  \in \,^{(l,\mu)}W \cap W^{(m,\nu)} 
\end{align}
by induction on the order $\succeq$ on $^{(l,\mu)}W \cap W^{(m,\nu)} $. 

If $v$ is minimal in $^{(l,\mu)}W \cap W^{(m,\nu)} $, 
it is also minimal in $O(v)$. 
Then we have $v \in \,^{(l,\mu)}W^{(m,\nu)}$, 
and  \eqref{H Tv} is clear.

Suppose that $v$ is not minimal in $^{(l,\mu)}W \cap W^{(m,\nu)} $.  
If $v$ is minimal in $O(v)$, 
we have $v \in \,^{(l,\mu)}W^{(m,\nu)}$, 
and  \eqref{H Tv} is clear. 
We also suppose that $v$ is not minimal in $O(v)$.  
Then we see that 
there exists $s_{j'} = x s_j x^{-1} \in  S_{\tau(x)} = S_{(l,\mu)} \cap x S_{(m,\nu)} x^{-1} $such that 
$a_j > a_{j+1}$ and $j, j+1 \in I(x)$ 
by definitions (see \eqref{W-y I}, \eqref{W act St I} and \eqref{W Def order succ}). 
Since $x \in \,^{(l,\mu)} \fS_n^{(m,\nu)}$, 
we have 
$\ell(s_{j'} x)= \ell(s_{j'}) + \ell(x)$ and $\ell(x s_j)= \ell (x) + \ell(s_j)$. 
Thus we have 
$ T_{j'} T_x = T_{s_{j'} x} = T_{x s_j} = T_x T_j$. 
Put $a_i =0$ if $i \not\in I(x)$. 
Then we have 
\begin{align}
\label{HTj' Tv}
\begin{split}
T_{j'} T_v 
&= T_{j'} (T_x  L_1^{a_1} L_2^{a_2} \dots L_n^{a_n})
\\
&= T_x T_j ( L_1^{a_1} L_2^{a_2} \dots L_n^{a_n})
\\
&= T_x (L_1^{a_1} \dots L_{j-1}^{a_{j-1}}) 
	(L_j^{a_{j+1}} L_{j+1}^{a_j} ) (L_{j+2}^{a_{j+2}} \dots L_n^{a_n}) T_j 
	\\ & \quad 
	- (q-1) \sum_{c= a_{j+1}}^{a_j -1} 
		T_x (L_1^{a_1} \dots L_{j-1}^{a_{j-1}})  (L_j^c L_{j+1}^{a_j + a_{j+1}-c} ) 
			(L_{j+2}^{a_{j+2}} \dots L_n^{a_n}),
\end{split}
\end{align}
where we use Lemma \ref{H Lemma L T} in the last equation 
(note $a_j > a_{j+1}$). 
Since 
\begin{align*}
&x (t_1^{a_1} \dots t_{j-1}^{a_{j-1}}) ( t_j^{a_{j+1}} t_{j+1}^{a_j} ) (t_{j+2}^{a_{j+2}} \dots t_n^{a_n}) 
	\prec v, 
\\
&x (t_1^{a_1} \dots t_{j-1}^{a_{j-1}}) (t_j^{c} t_{j+1}^{a_j + a_{j+1}-c}  )
	 (t_{j+2}^{a_{j+2}} \dots t_n^{a_n})  \prec v 
	 \quad (a_{j + 1} \leq c \leq a_j-1), 
\end{align*}
we have 
\begin{align*}
&T_x (L_1^{a_1} \dots L_{j-1}^{a_{j-1}}) 
	(L_j^{a_{j+1}} L_{j+1}^{a_j} ) (L_{j+2}^{a_{j+2}} \dots L_n^{a_n}) 
	\in \sum_{ u \in \,^{(l,\mu)} W^{(m,\nu)}} \sH_{(l,\mu)} T_u \sH_{(m,\nu)} , 
\\
& T_x (L_1^{a_1} \dots L_{j-1}^{a_{j-1}})  (L_j^c L_{j+1}^{a_j + a_{j+1}-c} ) 
			(L_{j+2}^{a_{j+2}} \dots L_n^{a_n})
	\in \sum_{ u \in \,^{(l,\mu)} W^{(m,\nu)}} \sH_{(l,\mu)} T_u \sH_{(m,\nu)} , 
\end{align*}
by the assumption of the induction. 
Together with \eqref{HTj' Tv}, we have 
\begin{align*}
T_v \in \sum_{ u \in \,^{(l,\mu)} W^{(m,\nu)}} \sH_{(l,\mu)} T_u \sH_{(m,\nu)}, 
\end{align*} 
where we note that $T_{j'}^{-1} \in \sH_{(l,\mu)}$ and $T_j \in \sH_{(m,\nu)}$. 
Thus we proved \eqref{H Tv}. 

In order to prove (\roi), 
it is enough to show that 
\begin{align}
\label{H Tw in}
T_w \in \sum_{ u \in \,^{(l,\mu)} W^{(m,\nu)}} \sH_{(l,\mu)} T_u \sH_{(m,\nu)}  
\text{ for any } 
w = x t_1^{a_1} t_2^{a_2} \dots t_n^{a_n} \in W_{n,r}. 
\end{align}
We prove \eqref{H Tw in} by the induction on $\ell(x)$. 

If $\ell(x)=0$, we have 
\begin{align}  
\label{H Tw x=0}
T_w = L_1^{a_1} L_2^{a_2} \dots L_n^{a_n} 
= \begin{cases} 
	(L_1^{a_1} L_2^{a_2} \dots L_l^{a_l}) (L_{l+1}^{a_{l+1}}  L_{l+2}^{a_{l+2}} \dots L_n^{a_n}) 
		& \text{ if } l \geq m, 
	\\
	(L_{m+1}^{a_{m+1}} L_{m+2}^{a_{m+2}} \dots L_n^{a_n}) 
		(L_1^{a_1} L_2^{a_2} \dots L_{m}^{a_m}) 
		& \text{ if } l <m. 
\end{cases}
\end{align}
We see that 
\begin{align*}
&(t_{l+1}^{a_{l+1}}  t_{l+2}^{a_{l+2}} \dots t_n^{a_n})  \in \,^{(l,\mu)}W \cap W^{(m,\nu)} 
	\text{ if } l \geq m, 
\\
& (t_{m+1}^{a_{m+1}} t_{m+2}^{a_{m+2}} \dots t_n^{a_n}) \in \,^{(l,\mu)}W \cap W^{(m,\nu)} 
	\text{ if } l <m, 
\end{align*}
where we note that 
$I(e) =[m+1, n] \cap \{l+1, l+2, \dots, n\}$ 
for the identity element $e \in \fS_n$. 
Thus we have 
\begin{align*}
&(L_1^{a_1} L_2^{a_2} \dots L_l^{a_l}) \in \sH_{(l,\mu)}, 
\\
&(L_1^{a_1} L_2^{a_2} \dots L_{m}^{a_m})  \in \sH_{(m,\nu)}, 
\\
&(L_{l+1}^{a_{l+1}}  L_{l+2}^{a_{l+2}} \dots L_n^{a_n})
	\in \sum_{ u \in \,^{(l,\mu)} W^{(m,\nu)}} \sH_{(l,\mu)} T_u \sH_{(m,\nu)} 
	\text{ if } l \geq m, 
\\
& (L_{m+1}^{a_{m+1}} L_{m+2}^{a_{m+2}} \dots L_n^{a_n})
	\in \sum_{ u \in \,^{(l,\mu)} W^{(m,\nu)}} \sH_{(l,\mu)} T_u \sH_{(m,\nu)} 
	\text{ if } l < m
\end{align*}
by \eqref{H Tv}. 
Then, together with \eqref{H Tw x=0}, we have 
$T_w \in \sum_{ u \in \,^{(l,\mu)} W^{(m,\nu)}} \sH_{(l,\mu)} T_u \sH_{(m,\nu)} $. 

Suppose that $\ell(x) >0$. 
We can uniquely write 
$x = x_1 x_2 x_3$ 
for some 
$x_2 \in \,^{(l,\mu)} \fS^{(m,\nu)}$, $x_1 \in (\fS_{(l,\mu)})^{\tau(x_2)}$ 
and $ x_3 \in \fS_{(m,\nu)}$. 
Then we have $\ell(x) = \ell(x_1) + \ell(x_2) + \ell(x_3)$ by the general theory of Coxeter group. 
Thus we have 
\begin{align*}
T_w 
= T_{x_1} T_{x_2} T_{x_3} L_1^{a_1} L_2^{a_2} \dots L_n^{a_n} 
= T_{x_1} T_{x_2} T_{x_3} (L_{m+1}^{a_{m+1}} L_{m+2}^{a_{m+2}} \dots L_n^{a_n}) 
	(L_1^{a_1} L_2^{a_2} \dots L_m^{a_m}). 
\end{align*}
Applying Lemma \ref{H Lemma Tx L},  
we have 
\begin{align}
\label{H Tw decom}
\begin{split}
T_w 
&= T_{x_1} T_{x_2} (L_{x_3(m+1)}^{a_{m+1}} L_{x_3(m+2)}^{a_{m+2}} \dots L_{x_3(n)}^{a_n}) 
	T_{x_3} (L_1^{a_1} L_2^{a_2} \dots L_m^{a_m})
	\\ & \quad 
	+ \sum_{y_3 < x_3}  \sum_{  (b_{m+1}, \dots b_n) \in [0,r-1]^{n-m} } r_{y_3}^{(b_{m+1}, \dots, b_n)} 
		T_{x_1} T_{x_2} T_{y_3} (L_{m+1}^{b_{m+1}} L_{m+2}^{b_{m+2}} \dots L_n^{b_n}) 
		(L_1^{a_1} L_2^{a_2} \dots L_m^{a_m}). 
\end{split}
\end{align}
Since $T_{x_1} T_{x_2} T_{x_3} = T_{x}$, 
we have 
\begin{align}
\label{H sum in}
&\sum_{y_3 < x_3}  \sum_{  (b_{m+1}, \dots b_n) \in [0,r-1]^{n-m} } r_{y_3}^{(b_{m+1}, \dots, b_n)} 
		T_{x_1} T_{x_2} T_{y_3} (L_{m+1}^{b_{m+1}} L_{m+2}^{b_{m+2}} \dots L_n^{b_n}) 
		(L_1^{a_1} L_2^{a_2} \dots L_m^{a_m})
 \\ & \hspace*{2cm}
 \in \sum_{ u \in \,^{(l,\mu)} W^{(m,\nu)}} \sH_{(l,\mu)} T_u \sH_{(m,\nu)}  \nonumber
\end{align} 
by the assumption of the induction. 
Note that $\{x_3(m+1), x_3(m+2), \dots, x_3(n)\} =[m+1,n]$ by $x_3 \in \fS_{(m,\nu)}$. 
Applying Lemma \ref{H Lemma Tx L}, 
we have 
\begin{displaymath}
\begin{split}
T_{x_1} T_{x_2} (L_{x_3(m+1)}^{a_{m+1}} & L_{x_3(m+2)}^{a_{m+2}} \dots L_{x_3(n)}^{a_n}) \\
&= T_{x_1} T_{x_2} (L_{m+1}^{a'_{m+1}} L_{m+2}^{a'_{m+2}} \dots L_{n}^{a'_n}) 
\\
&= T_{x_1} (L_{x_2(m+1)}^{a'_{m+1}} L_{x_2(m+2)}^{a'_{m+2}} \dots L_{x_2 (n)}^{a'_n}) T_{x_2} 
	\\  
	&\qquad + \sum_{y_2 < x_2} \sum_{(b_1, \dots, b_n) \in [0,r-1]^n} 
		r_{y_2}^{(b_1, \dots, b_n)} T_{x_1} T_{y_2} (L_1^{b_1} L_2^{b_2} \dots L_n^{b_n}), 
\end{split}
\end{displaymath}
where we put $a'_i = a_{x_3^{-1}(i)}$ for $i=m+1,\dots,n$. 
We have 
\[
\sum_{y_2 < x_2} \sum_{(b_1, \dots, b_n) \in [0,r-1]^n} 
		r_{y_2}^{(b_1, \dots, b_n)} T_{x_1} T_{y_2} (L_1^{b_1} L_2^{b_2} \dots L_n^{b_n}) 
\in \sum_{ u \in \,^{(l,\mu)} W^{(m,\nu)}} \sH_{(l,\mu)} T_u \sH_{(m,\nu)}  
\]
by the assumption of the induction. 
Applying Lemma \ref{H Lemma Tx L}, we also have 
\begin{align*}
T_{x_1} (L_{x_2(m+1)}^{a'_{m+1}} & L_{x_2(m+2)}^{a'_{m+2}} \dots L_{x_2 (n)}^{a'_n}) T_{x_2} 
= T_{x_1} (\prod_{m+1 \leq i \leq n \atop x_2(i) \leq l} L_{x_2(i)}^{a'_i}) 
	( \prod_{m+1 \leq i \leq n \atop x_2(i) >l} L_{x_2(i)}^{a'_i}) T_{x_2} 
\\
&= T_{x_1} (\prod_{m+1 \leq i \leq n \atop x_2(i) \leq l} L_{x_2(i)}^{a'_i}) 
	T_{x_2} ( \prod_{ i \in I(x_2)} L_i^{a'_i}) 
	\\ & \qquad 
	- T_{x_1} (\prod_{m+1 \leq i \leq n \atop x_2(i) \leq l} L_{x_2(i)}^{a'_i}) 
	(\sum_{y_2 < x_2} \sum_{(b_1, \dots, b_n) \in [0,r-1]^n} 
		r_{y_2}^{(b_1, \dots, b_n)} 
		T_{y_2} (L_1^{b_1}L_2^{b_2} \dots L_n^{b_n}))
\\
& \qquad \in \sum_{ u \in \,^{(l,\mu)} W^{(m,\nu)}} \sH_{(l,\mu)} T_u \sH_{(m,\nu)}  
\end{align*}
since 
$T_{x_1} (\prod_{m+1 \leq i \leq n \atop x_2(i) \leq l} L_{x_2(i)}^{a'_i}) \in \sH_{(l,\mu)}$, 
$T_{x_2} (\prod_{i \in I(x_2)} L_i^{a'_i}) 
\in \sum_{ u \in \,^{(l,\mu)} W^{(m,\nu)}} \sH_{(l,\mu)} T_u \sH_{(m,\nu)} $ 
by \eqref{H Tv}, 
and 
$T_{y_2}(L_1^{b_1}L_2^{b_2} \dots L_n^{b_n}) 
\in \sum_{ u \in \,^{(l,\mu)} W^{(m,\nu)}} \sH_{(l,\mu)} T_u \sH_{(m,\nu)}$ 
for $y_2 < x_2$ 
by the assumption of the induction. 
As a consequence, we have 
\begin{align*}
T_{x_1} T_{x_2} (L_{x_3(m+1)}^{a_{m+1}} L_{x_3(m+2)}^{a_{m+2}} \dots L_{x_3(n)}^{a_n}) 
\in \sum_{ u \in \,^{(l,\mu)} W^{(m,\nu)}} \sH_{(l,\mu)} T_u \sH_{(m,\nu)}, 
\end{align*}
and this implies that 
\begin{align}
\label{H Tx1 Tx2}
T_{x_1} T_{x_2} (L_{x_3(m+1)}^{a_{m+1}} L_{x_3(m+2)}^{a_{m+2}} \dots L_{x_3(n)}^{a_n}) 
T_{x_3} (L_1^{a_1} L_2^{a_2} \dots L_m^{a_m})
\in \sum_{ u \in \,^{(l,\mu)} W^{(m,\nu)}} \sH_{(l,\mu)} T_u \sH_{(m,\nu)}, 
\end{align}
since $T_{x_3} (L_1^{a_1} L_2^{a_2} \dots L_m^{a_m}) \in \sH_{(m,\nu)}$. 
Thanks to \eqref{H Tw decom},  \eqref{H sum in} and \eqref{H Tx1 Tx2}, 
we obtain \eqref{H Tw in}, 
and we proved (\roi). 

We prove (\roii). 
For each $u \in \,^{(l,\mu)} W ^{(m,\nu)}$, 
we see that 
\begin{align*} 
\{T_{w_1} T_v \mid w_1 \in (W_{(l,\mu)})^{(k(u), \pi(u))}, \, v \in W_{(k(u),\pi(u))}\}
\end{align*} 
is an $R$-free basis of $\sH_{(l,\mu)}$ by Proposition \ref{H Prop Hnr Hlmu}. 
Note that  $T_u \sH_{(m,\nu)}$ is a left $\sH_{(k(u),\pi(u))}$-module by Proposition \ref{H Prop TuHmnu}, 
then (\roi) implies that 
$\sH_{n,r}$ is spanned by 
\begin{align*}
\{ T_{w_1} T_u T_{w_2}
	 \mid u \in \,^{(l,\mu)} W^{(m,\nu)}, \, w_1 \in (W_{(l,\mu)})^{(k(u),\pi(u))}, \, w_2 \in W_{(m,\nu)} \} 
= \{\wt{T}_w \mid w \in W_{n,r}\} 
\end{align*}
as an $R$-module. 
Then  
we can define the surjective homomorphism of $R$-modules 
$\phi : \sH_{n,r} \ra \sH_{n,r}$ such that $\phi(T_w) = \wt{T}_w$ ($w \in W_{n,r}$), 
and $\phi$ is an isomorphism by \cite[Theorem 2.4]{Mats}. 
Thus $\{\wt{T}_w \mid w \in W_{n,r}\}$ is an $R$-free basis of $\sH_{n,r}$. 
\end{proof}


\para \label{H Def functors}
For a parabolic subgroup $W_{(l,\mu)}$ of $W_{n,r}$ 
(resp. $W_{(k(u), \pi^{\sharp}(u))}$ of $W_{(m,\nu)}$), 
we define the restriction functor 
\begin{align*} 
&\HRes^{W_{n,r}}_{W_{(l,\mu)}} : \sH_{n,r} \cmod \ra \sH_{(l,\mu)} \cmod
\\
& 
(\text{resp. } 
 \HRes^{W_{(m,\nu)}}_{W_{(k(u), \pi^{\sharp}(u))}}
 : \sH_{(m,\nu)}
\cmod
 \ra \sH_{(k(u), \pi^{\sharp}(u))}
\cmod
 )
\end{align*} 
by the restriction of the action. 
We also define the induction functor 
\begin{align*}
\HInd^{W_{n,r}}_{W_{(l,\mu)}} = \sH_{n,r} \otimes_{\sH_{(l,\mu)}} - : \sH_{(l,\mu)} \cmod \ra \sH_{n,r} \cmod,
\end{align*} 
where we regard $\sH_{n,r}$ as an $(\sH_{n,r}, \sH_{(l,\mu)})$-bimodule by multiplications. 

For $u \in \,^{(l,\mu)} W^{(m,\nu)}$,  
we define the induction functor
\begin{align*} 
\HInd^{W_{(l,\mu)}}_{W_{(k(u), \pi(u))}} 
	= \sH_{(l,\mu)} \otimes_{\sH_{(k(u),\pi(u))}} - :
 \sH_{(k(u),\pi(u))} \cmod \ra \sH_{(l,\mu)}
\cmod
 , 
\end{align*}
where we note that $\sH_{(k(u), \pi(u))}$ is a subalgebra of $\sH_{(l,\mu)}$. 
We also define the functor  $T_u(-) : \sH_{(k(u), \pi^{\sharp}(u))} \cmod \ra \sH_{(k(u), \pi(u))} \cmod$ by 
\begin{align*} 
T_u(-)= T_u \sH_{(k(u), \pi^{\sharp}(u))} \otimes_{\sH_{(k(u), \pi^{\sharp}(u))}} - 
	: \sH_{(k(u), \pi^{\sharp}(u))} \cmod \ra \sH_{(k(u), \pi(u))} \cmod. 
\end{align*} 
We see that any functor defined in the above is exact by Proposition \ref{H Prop Hnr Hlmu}. 
Then now we obtain the first main theorem of this paper.

\begin{thm}[The Mackey formula for cyclotomic Hecke algebras]
\label{H Thm Mackey}
For $0\leq l,m \leq n$, $\mu \vDash n-l$ and $\nu \vDash n-m$, we have the following: 
\begin{enumerate} 
\item 
There exists an isomorphism of $(\sH_{(l,\mu)}, \sH_{(m,\nu)})$-bimodules 
\begin{align*}
\sH_{n,r} \ra \bigoplus_{ u \in \, ^{(l,\mu)}W^{(m,\nu)}} 
	(\sH_{(l,\mu)} \otimes_{\sH_{(k(u), \pi(u))}} T_u \sH_{(m,\nu)})
\end{align*}
such that 
$\wt{T}_w =T_{w_1} T_u T_{w_2} \mapsto T_{w_1} \otimes T_u T_{w_2}$ 
($u \in \,^{(l,\mu)} W^{(m,\nu)}$, $w_1 \in (W_{(l,\mu)})^{(k(u),\pi(u))}$, $w_2 \in W_{(m,\nu)}$). 

\item 
For a left $\sH_{(m,\nu)}$-module $M$, we have a natural isomorphism of left $\sH_{(l,\mu)}$-modules 
\begin{align*}
\sH_{n,r} \otimes_{\sH_{(m,\nu)}} M 
\cong 
\bigoplus_{u \in \, ^{(l,\mu)}W^{(m,\nu)}} 
	(\sH_{(l,\mu)} \otimes_{\sH_{(k(u), \pi(u))}} T_u \sH_{(m,\nu)}) \otimes_{\sH_{(m,\nu)}} M.
\end{align*}

\item 
We have an isomorphism of functors 
\begin{align*}
\HRes_{W_{(l,\mu)}}^{W_{n,r}} \circ \HInd_{W_{(m,\nu)}}^{W_{n,r}} 
\cong 
\bigoplus_{u \in \,^{(l,\mu)}W^{(m,\nu)}} 
	\HInd^{W_{(l,\mu)}}_{W_{(k(u), \pi(u))}} \circ T_u(-) \circ \HRes^{W_{(m,\nu)}}_{W_{(k(u),\pi^{\sharp}(u))}}. 
\end{align*}
\end{enumerate}
\end{thm}

\begin{proof} 
We prove (\roi). 
Since $\{\wt{T}_w \mid w \in W_{n,r}\}$ is an $R$-free basis of $\sH_{n,r}$ by Proposition \ref{H H decom} (\roii), 
we can define a homomorphism of $R$-modules 
\begin{align*}
\Phi : \sH_{n,r}  \ra \bigoplus_{ u \in \, ^{(l,\mu)}W^{(m,\nu)}} 
	(\sH_{(l,\mu)} \otimes_{\sH_{(k(u), \pi(u))}} T_u \sH_{(m,\nu)}), 
\end{align*}
by $\wt{T}_w = T_{w_1} T_u T_{w_2} \mapsto T_{w_1} \otimes T_u T_{w_2}$ 
($u \in \,^{(l,\mu)} W^{(m,\nu)}$, $w_1 \in (W_{(l,\mu)})^{(k(u),\pi(u))}$, $w_2 \in W_{(m,\nu)}$). 
In order to define the inverse map of $\Phi$, 
for $u \in \,^{(l,\mu)}W^{(m,\nu)}$, 
let 
\begin{align*}
\Psi'_u : \sH_{(l,\mu)} \times T_u \sH_{(m,\nu)} \ra \sH_{n,r}, 
\end{align*}
be the multiplication map in $\sH_{n,r}$. 
Since $T_u \sH_{(m,\nu)}$ (resp. $\sH_{(l,\mu)}$) is a left (resp. right) $\sH_{(k(u),\pi(u))}$-module 
by multiplications in $\sH_{n,r}$ (see Proposition \ref{H Prop TuHmnu}), 
it is clear that $\Psi'_u$ is a $\sH_{(k(u),\pi(u))}$-balanced map. 
Thus we have the homomorphism of $R$-modules 
\begin{align*}
\Psi_u : \sH_{(l,\mu)} \otimes_{\sH_{(k(u),\pi(u))}} T_u \sH_{(m,\nu)} \ra \sH_{n,r}, 
\quad 
X \otimes Y \mapsto XY. 
\end{align*}
Then it is clear that 
$\Psi = \bigoplus_{ u \in \,^{(l,\mu)}W^{(m,\nu)}} \Psi_u$ 
is the inverse map of $\Phi$, 
and we see that $\Phi$ is isomorphism. 
It is clear that 
$\Psi$ is an isomorphism of $(\sH_{(l,\mu)}, \sH_{(m,\nu)})$-bimodules 
since actions in both sides are given by multiplications. 
Thus $\Phi$ is also an isomorphism of $(\sH_{(l,\mu)}, \sH_{(m,\nu)})$-bimodules. 
(\roii) follows from (\roi), 
and (\roiii) follows from (\roi) together with  Corollary \ref{H Cor Tu func}. 
\end{proof}



\section{The categories $\cO$ of rational Cherednik algebras} 
\label{category O}
In this section, we review some known results for categories $\cO$ of 
rational Cherednik algebras, and we also prepare some properties for
the next section. 

\para
Let $W$ be a finite complex reflection group and let $\frh$ be the $\C$-vector space on 
which $W$ acts by reflections. 
Let $\sA_W$ be the set of reflection hyperplanes,
and let $\frh_{W}^{reg} = \frh \backslash \bigcup_{H \in \sA_W} H$ be its complement.
We denote by $\cS_W$ the set of reflections in $W$. For $s \in \cS_W$, write $\lambda_s$
for the non-trivial eigenvalue of $s$ in $\frh^*$. For $s \in \cS_W$, let $\alpha_s \in \frh^*$ be
a generator of $\Image(s |_{\frh^*} - 1)$ and let $\alpha_s^{\vee}$ be the
generator of $\Image(s |_{\frh} - 1)$ such that 
$\langle \alpha_s, \alpha_s^{\vee} \rangle = 2$ where $\langle \;\;, \;\;\rangle$
is the standard pairing between $\frh$ and $\frh^*$.
Let $\calD(\frh_{W}^{reg})$ be the $\C$-algebra of algebraic differential operators 
on the smooth affine manifold $\frh_{W}^{reg}$. 
The action of the group $W$ on $\frh$ induces an action of $W$
on the $\C$-algebra $\calD(\frh_{W}^{reg})$. We denote the smash product of the algebra 
$\calD(\frh_{W}^{reg})$ and the group $W$ by $\calD(\frh_{W}^{reg}) \rtimes W$.
The rational Cherednik algebra $H(W) = H(W, \frh)$ associated with $W$ is a subalgebra of 
$\calD(\frh_{W}^{reg}) \rtimes W$ which is generated by elements of $\C[\frh]$,
elements of $W$ and the Dunkl operators $D_{\xi}$ for $\xi \in \frh$:
\[
 D_{\xi} = \partial_{\xi} + \sum_{s \in \cS_W} \frac{2 c_s}{1 - \lambda_s} 
 \frac{\alpha_s(\xi)}{\alpha_s} (s - 1) \in \calD(\frh_{W}^{reg}) \rtimes W
\]
where $\{c_s\}_{s \in \cS_W}$ is the parameter of the algebra $H(W)$. 

\para
The category $\cO(W)$ is a full subcategory of the category of finitely generated
$H(W)$-modules on which object the Dunkl operators acts locally
nilpotently.  For a module $M \in \cO(W)$, we consider the localization
$M^{an} = \cO^{an}_{\frh_{W}^{reg}} \otimes_{\C[\frh]} M$ where $\cO^{an}_{\frh_{W}^{reg}}$ is the
sheaf of holomorphic functions on $\frh_{W}^{reg}$. Since we have 
$\C[\frh_{W}^{reg}] \otimes_{\C[\frh]} H(W) = \calD(\frh_{W}^{reg}) \rtimes W$, the algebra
$\calD(\frh_{W}^{reg}) \rtimes W$ acts on $M^{an}$. Namely, $M^{an}$ is a vector bundle on 
$\frh_{W}^{reg}$ with a $W$-equivariant flat connection. Let $(M^{an})^{\nabla}$ be
the $W$-equivariant local system of horizontal sections of $M^{an}$. For any point
$p \in \frh_{W}^{reg}$, the stalk $(M^{an})^{\nabla}_{p}$ at the point $p$ is a 
finite-dimensional vector space over $\C$. Fix a point $p_0 \in \frh_{W}^{reg}$, and then 
we set $\KZ_W(M) = (M^{an})^{\nabla}_{p_0}$ as a vector space. 

We review an action of a Hecke algebra on $\KZ_W(M) = (M^{an})^{\nabla}_{p_0}$ defined by
monodromy. First note that the $W$-equivariant local system $(M^{an})^{\nabla}$ on
$\frh_{W}^{reg}$ is naturally identified with a local system on $\frh_{W}^{reg} / W$. Thus,
we have an action of the fundamental group $\pi_1(\frh_{W}^{reg} / W, \bar{p}_0)$ on
$\KZ_W(M) = (M^{an})^{\nabla}_{p_0}$ via monodromy representation, where $\bar{p}_0$ is
the image of $p_0 \in \frh_{W}^{reg}$ on $\frh_{W}^{reg} / W$. By
\cite{GGOR}, this action factors
through a Hecke algebra $\sH(W)$ associated with $W$ with a parameter $q$ determined by
the formula in \cite[Section 5.2]{GGOR}.

Let $[0,1]$ be the real closed interval between $0$ and $1$ (not an
integer interval).
For a path $\gamma : [0, 1] \longrightarrow \frh_{W}^{reg}$ and a germ
$v \in (M^{an})^{\nabla}_{\gamma(0)}$ at $\gamma(0)$, we have its analytic continuation 
$v' \in (M^{an})^{\nabla}_{\gamma(1)}$ at $\gamma(1)$ through the path $\gamma$.
Then, we define an operator of analytic continuation
\[
 S_M(\gamma): (M^{an})^{\nabla}_{\gamma(0)} \longrightarrow (M^{an})^{\nabla}_{\gamma(1)},
 \qquad v \mapsto v'.
\]
Following
\cite[(2.10)]{BMR}, recall how we obtain a homomorphism of $\pi_1$ to $W$:
Note that, for a loop $\sigma \in \pi_1(\frh_{W}^{reg} / W, \bar{p}_0)$ and 
a point $p \in W p_0$,
we have a unique path ${}^p \widetilde{\sigma} : [0,1] \longrightarrow \frh_{W}^{reg}$
such that ${}^p\widetilde{\sigma}(0) = p$ and its image in $\frh_{W}^{reg} / W$ coincides
with $\sigma$.
The path $\widetilde{\sigma}$ in $\frh^{reg}$ is called a lift of
the element $\sigma \in \pi_1(\frh^{reg}/W, \bar{p}_0)$. 
As \cite{BMR}, we describe elements of $\pi_1(\frh^{reg}_{W} / W, \bar{p}_0)$
by their lifts (see \cite[Appendix A]{BMR}).
For the above loop $\sigma$, we set $\overline{\sigma} = w \in W$ where $w$ is an
element with ${}^p \widetilde{\sigma}(1) = w(p)$. Then,
 we have a homomorphism of groups (\cite[(2.10)]{BMR})
\[
 \overline{(-)} : \pi_1(\frh_{W}^{reg}/W, \bar{p}_0)^{\operatorname{opp}} \longrightarrow W.
\]
Now we have an action of the braid group given by monodromy
\[
 \widetilde{T}_M: \pi_1(\frh_{W}^{reg} / W, \bar{p}_0)^{\operatorname{opp}} \longrightarrow GL((M^{an})^{\nabla}_{p_0}),
\qquad \sigma \mapsto S_M({}^{p_0}\widetilde{\sigma}^{-1}) \overline{\sigma}.
\]
By \cite[Theorem~4.12]{BMR} and \cite[Theorem 5.13]{GGOR}, 
the
linearly extended
homomorphism $\widetilde{T}_M$ factors through an
algebra homomorphism $\widetilde{T}_M : \sH(W) \longrightarrow \gEnd_{\C}((M^{an})_{p_0}^{\nabla})$,
and thus $\KZ_W(M)$ is equipped with the action of the Hecke algebra $\sH(W)$.
Then we have the functor 
\begin{align*}
\KZ_W : \cO(W) \ra \sH(W)\cmod, 
\quad 
M \mapsto \KZ_W(M).
\end{align*}
By \cite[Section 5.4]{GGOR}, we have the following proposition: 


\begin{prop}
\label{Prop KZ}
\begin{enumerate} 
\item 
There exists a projective object $P_{\KZ_W} \in \cO$ such that 
$\sH (W) \cong \operatorname{End}_{\cO} (P_{\KZ_W})^{\operatorname{opp}}$ as algebras. 

\item 
We have the isomorphism of functors 
$\KZ_W \cong \operatorname{Hom}_{\cO}(P_{\KZ_W}, -)$. 

\item 
$\KZ_W$ is fully faithful on projective objects. 

\end{enumerate} 
\end{prop}
\begin{proof} 
(\roi) and (\roii) were proved in \cite[Theorem 5.15]{GGOR}. 
(\roiii) follows from \cite[Theorem 5.16]{GGOR} and \cite[Proposition 4.33]{Rouq}. 
\end{proof}

\para 
For a parabolic subgroup $W'$ of $W$, 
Bezrukavnikov and Etingof introduced the functors of parabolic restriction 
$\ORes^{W}_{W'}$ and induction $\OInd^{W}_{W'}$ for modules of the category $\cO$
in \cite{BE}. They are exact functors 
$\ORes^{W}_{W'} : \cO(W) \longrightarrow \cO(W')$,
$\OInd^{W}_{W'} : \cO(W') \longrightarrow \cO(W)$ between the categories $\cO$
for the rational Cherednik algebras $H(W)$ and $H(W')$. 
For these functors $\ORes^{W}_{W'}$ and $\OInd^{W}_{W'}$, 
we have the following properties. 

\begin{prop}
\label{prop:ORes-KZ-props}
For  a parabolic subgroup $W'$ of $W$, we have
\begin{enumerate}
 \item The functors $\ORes^{W}_{W'}$, $\OInd^{W}_{W'}$ are exact,
       and they send projective objects to projective objects.
 \item There exists an isomorphism of functors 
       $\KZ_{W'} \circ \ORes^{W}_{W'} \cong \HRes^{W}_{W'} \circ \KZ_W$.
 \item There exists an isomorphism of functors 
       $\KZ_{W} \circ \OInd^{W}_{W'} \cong \HInd^{W}_{W'} \circ \KZ_{W'}$.
\end{enumerate}
\begin{proof}
 (i)
 By \cite[Proposition~3.9]{BE}, the functors $\ORes^{W}_{W'}$, $\OInd^{W}_{W'}$ are
 exact.  These functors are biadjoint with each others by the consequence
 of \cite[Theorem~3.10]{BE} and \cite[Proposition~2.9]{Shan} (see also \cite{Losev}).
 By \cite[Proposition~2.3.10]{Weibel}, an additive functor being left adjoint to an
 exact functor sends projectives to projectives. Thus, $\ORes^{W}_{W'}$ and
 $\OInd^{W}_{W'}$ send projectives to projectives.

 (ii) This is stated in \cite[Theorem~2.1]{Shan}.  (iii) This is also
 stated in \cite[Corollary~2.3]{Shan}.
\end{proof}
\end{prop}

\para 
\label{twist functor}
For a parabolic subgroup $W'$ of $W$ and an element $x \in W$,
we have a $\C$-algebra isomorphism $\theta_{W'}^{(x)} : H(x W' x^{-1}) \longrightarrow H(W' )$
given by $f \mapsto x^{-1} f x$ for $f \in H( x W' x^{-1})$. 
We define a functor 
\begin{align*} 
\Theta^{(x)}_{W'} : \cO( W' ) \ra \cO( x W' x^{-1}), 
\quad 
M \mapsto M^{\theta^{(x)}_{W'}}, 
\end{align*} 
where $M^{\theta^{(x)}_{W'}} =M$ as vector spaces and the action is twisted by $\theta^{(x)}_{W'}$. 
From definitions, 
we have the following proposition:  


\begin{prop}
\label{Prop Theta}
For  parabolic subgroups $W', W''$ of $W$  with $W'' \subset W'$ and elements $x,y \in W$,
we have the following $:$
\begin{enumerate} 
\item 
$\Theta^{(x^{-1})}_{x W' x^{-1}}$ is a quasi-inverse of $\Theta^{(x)}_{W'} $. 

\item 
We have the isomorphism of functors 
$\Theta^{(x)}_{y W' y^{-1}} \circ \Theta^{(y)}_{W'} \cong \Theta^{(xy)}_{W'}$. 

\item 
\begin{enumerate} 
\item 
We have the isomorphism of functors 
$\Theta^{(x)}_{W'} \circ \ORes^{W}_{W'} \cong \ORes^{W}_{x W' x^{-1}}$. 

\item 
We have the isomorphism of functors 
$\OInd^{W}_{W'} \circ \Theta^{(x^{-1})}_{x W' x^{-1}} \cong \OInd^{W}_{x W' x^{-1}}$. 

\end{enumerate} 

\item 
\begin{enumerate} 
\item 
We have the isomorphism of functors 
$\Theta^{(x)}_{W''} \circ \ORes^{W'}_{W''} \cong \ORes^{x W' x^{-1}}_{x W'' x^{-1}} \circ \Theta^{(x)}_{W'}$. 

\item 
We have the isomorphism of functors 
$\Theta^{(x)}_{W''} \circ \OInd^{W'}_{W''} \cong \OInd^{x W' x^{-1}}_{x W'' x^{-1}} \circ \Theta^{(x)}_{W'}$. 
\end{enumerate}
\end{enumerate}
\end{prop}

\begin{proof}
(\roi) and (\roii) are clear from the definition. 
We can check (\roiii) (a) (resp. (\roiv) (a)) from definitions by direct calculation. 
(\roiii) (b) (resp. (\roiv) (b)) follows from (\roiii) (a) (resp. (\roiv) (a))  with the uniqueness of the adjoint functors.  
\end{proof}

Note also that, for a module $M \in \cO(W')$ and a point $p \in \frh^{reg}_{W'}$, 
the functor $\Theta^{(x)}_{W'}$ induces an isomorphism
$\widehat{\Theta}^{(x)}_{W'} : (M^{an})_{p}^{\nabla} \longrightarrow (\Theta^{(x)}_{W'} M^{an})_{p}^{\nabla} \simeq (M^{an})_{x(p)}^{\nabla}$ of vector spaces, and we have
the following commutative diagram:
\[
\xymatrix{
 (M^{an})_{\gamma(0)}^{\nabla} \ar[rr]^{S_M(\gamma)} \ar[d]_{\widehat{\Theta}^{(x)}_{W'}} 
 & & (M^{an})_{\gamma(1)}^{\nabla} \ar[d]^{\widehat{\Theta}^{(x)}_{W'}}\\
 (M^{an})_{x(\gamma)(0)}^{\nabla} \ar[rr]_{S_{\Theta^{(x)}_{W'}(M)}(x(\gamma))} & & (M^{an})_{x(\gamma)(1)}^{\nabla}
}
\]
for a path $\gamma$ in $\frh_{W'}^{reg}$. Here $x(\gamma)$ is the path given by 
$x(\gamma)(t) = x(\gamma(t))$.

Before closing this section, we introduce a way to let the Mackey decomposition functor between
module categories over Hecke algebras lift to the functor between categories
$\cO$ over rational Cherednik algebras.

\begin{prop}\label{prop:H-Mackey-RCA}
Let $W_a,W_b$  be standard parabolic subgroups of $W$.
  Suppose that there is a subset ${}^a D^b$ of $W$ and elements
  $\{T_u \in \sH(W)~|~u \in {}^a D^b\}$ such that
\begin{align}\label{H Assm Mackey}
\HRes_{W_a}^{W} \circ \HInd_{W_b}^{W} 
\cong 
\bigoplus_{u \in \,^{a}D^{b}} 
	\HInd^{W_{a}}_{W_a \cap u W_b u^{-1}} \circ T_u(-) \circ
 \HRes^{W_b}_{u^{-1} W_a u \cap W_b}
\end{align}
and
both  $W_a \cap u W_b u^{-1}$ and $u^{-1}W_a u \cap  W_b $ are standard for any $u \in {}^a D^b$ .
If for any $u \in {}^a D^b$
\begin{equation}\label{comKZ}
  \KZ_{W_a \cap u  W_b u^{-1}} \circ \Theta^{(u)}_{u^{-1}W_a u \cap W_b}
  \cong
  T_u (-) \circ \KZ_{u^{-1} W_a u \cap W_b},
\end{equation}
  then
  we have the following isomorphism of functors $:$
 \[
  \ORes^{W}_{W_a} \circ \OInd^{W}_{W_b}
 \cong \bigoplus_{u \in {}^a D^b} 
 \OInd^{W_a}_{W_a \cap u W_b u^{-1}} \circ\, \Theta^{(u)}_{u^{-1}W_a u \cap W_b} \circ
 \ORes^{W_b}_{u^{-1} W_a u \cap W_b}.
 \]
 \end{prop}
\begin{proof}
Set
\begin{align*}
 K &= \ORes^{W}_{W_a} \circ \OInd^{W}_{W_b},  
 \\
 L &=
 \OInd^{W_a}_{W_a \cap u W_b u^{-1}} \circ\, \Theta^{(u)}_{u^{-1}W_a u \cap W_b} \circ
 \ORes^{W_b}_{u^{-1} W_a u \cap W_b},\\
 ^{\sH}K &= \HRes^{W}_{W_a} \circ \HInd^{W}_{W_b},  
 \\
 ^{\sH}L &= \bigoplus_{u \in {}^a D^b}
 \HInd^{W_a}_{W_a \cap u W_b u^{-1}} \circ\, T_u (-) \circ
 \HRes^{W_b}_{u^{-1} W_a u \cap W_b}.
\end{align*}
By Proposition \ref{prop:ORes-KZ-props} (\roii), (\roiii) and the assumption (\ref{comKZ}), 
we see that 
$\KZ_{W_a} \circ K \cong \,^{\sH}K \circ \KZ_{W_b}$ 
and 
$\KZ_{W_a} \circ L \cong \,^{\sH}L \circ \KZ_{W_b}$.
Then 
the proposition follows from The assumption (\ref{H Assm Mackey}), 
Proposition \ref{Prop lifting} together with Proposition \ref{Prop KZ} 
and Proposition \ref{prop:ORes-KZ-props} (\roi). 
\end{proof}



\section{The Mackey formula for the categories $\cO$ of rational
Cherednik algebras of type $G(r,1,n)$} 
\label{SectMackey-O-Wnr}
In this section, we discuss the Mackey formula for the categories $\cO$ of the rational
Cherednik algebras associated with the complex reflection group $W_{n,r}$
of type $G(n, 1, r)$.

\para 
Consider the parabolic subgroups $W_{(l, \mu)}$ and $W_{(m, \nu)}$  of $W_{n,r}$.
For a double coset representative 
$u = x \prod_{i \in I(x)} t_i^{a_i} \in {}^{(l,\mu)} W^{(m, \nu)}$,
we have $W_{(l,\mu)} \cap u W_{(m, \nu)} u^{-1} = W_{(k(u), \pi(u))}$ by 
Proposition~\ref{W Prop Wkupiu}, and 
$u^{-1} W_{(k(u), \pi(u))} u = W_{(k(u), \pi^{\sharp}(u))}$ by \eqref{W u-1 Wkupiu u}.
Recall that we denote by $X_{(k(u), \pi(u))} \subset \{s_0, s_1, \dots, s_{n-1}\}$ 
(resp. $X_{(k(u), \pi^{\sharp}(u))}$) 
the set of standard generators of the parabolic subgroup $W_{(k(u), \pi(u))}$
(resp. $W_{(k(u), \pi^{\sharp}(u))}$).
We sometimes denote the functor $\Theta^{(u)}_{W_{(k(u), \pi^{\sharp}(u))}}$ by $u(-)$ 
and also denote $M^{\theta^{(u)}_{W_{(k(u), \pi^{\sharp}(u))}}}$ by $uM$ when 
we need not notify the subgroup $W_{(k(u), \pi^{\sharp}(u))}$.

\para
Let $\frh = \C^n$ be the reflection representation of the complex reflection group $W_{n,r}$
and fix a point $p_0 \in \frh^{reg}_{W_{n,r}}$.
Let $B_{n,r} = \pi_1(\frh^{reg}_{W_{n,r}} / W_{n,r}, \bar{p}_0)$ be the fundamental group
of $\frh^{reg}_{W_{n,r}} / W_{n,r}$, the braid group associated with $W_{n,r}$.
For $j=0$, $1$, $\dots$, $n-1$, we fix a generator 
$\sigma_j \in B_{n,r}$
of the braid group given in \cite[\S2B]{BMR} such that $\overline{\sigma}_j = s_j$.
Then the image of $\sigma_0$, $\dots$, $\sigma_{n-1}$ in $\sH_{n,r} = \sH(W)$ 
are $T_0$, $\dots$, $T_{n-1} \in \sH_{n,r}$, the generators of the Hecke algebra
$\sH_{n,r}$ which we introduced in \S\ref{H Section}. 
For $i=1$, $\dots$, $n$, we set 
$\gamma_i = \sigma_{i-1} \sigma_{i-2} \dots \sigma_1 \sigma_0 \sigma_1 \dots \sigma_{i-1}$,
an element in $B_{n,r}$. Then, its image in $\sH_{n,r}$
is $q^{i-1}L_i$ and we have $\overline{\gamma}_i = t_i$. Note that these elements $\gamma_1$, $\dots$,
$\gamma_n$ mutually commute since the commutativity of $t_1$, $\dots$, $t_n \in W_{n,r}$ 
is obtained only by using the braid relations.
For the double coset representative $u = x \prod_{i \in I(x)} t_i^{a_i} \in {}^{(l,\mu)}W^{(m,\nu)}$, we consider
an element 
$\omega = \bigl(\prod_{i \in I(x)} \gamma_i^{a_i} \bigr) \sigma_{i_l} \dots \sigma_{i_1}\in B_{n,r}$
where $x = s_{i_1} \dots s_{i_l}$ is a reduced expression of $x \in \fS_n$. Then, 
we have $\overline{\omega} = u$.  
By Proposition \ref{W prop only braid}, 
for $s_j \in X_{(k(u), \pi(u))}$, 
there exists $s_{\psi(j)} \in X_{(k(u), \pi^{\sharp}(u))}$ 
such that $s_j (s_{i_1} s_{i_2} \dots s_{i_l} \prod_{i \in I(x)} t_i^{a_i}) = (s_{i_1} s_{i_2} \dots s_{i_l} \prod_{i \in I(x)} t_i^{a_i}) s_{\psi(j)}$, 
and this identity implies the identity 
\begin{equation}
 \label{eq:braid us_j = s_j'u}
  \omega \sigma_j = \sigma_{\psi(j)} \omega
\end{equation}
in the braid group $B_{n,r}$.

By the embedding of \cite[\S2D]{BMR}, we identify the braid group $B_{(k(u), \pi(u))}$
(resp. $B_{(k(u), \pi^{\sharp}(u))}$) associated with the parabolic subgroup 
$W_{(k(u), \pi(u))}$ (resp. $W_{(k(u), \pi^{\sharp}(u))}$) with the subgroup 
of $B_{n,r}$ generated by the standard generators 
$\{ \sigma_j \,|\, s_j \in X_{(k(u), \pi(u))}\}$
(resp. $\{ \sigma_j \,|\, s_j \in X_{(k(u), \pi^{\sharp}(u))}\}$).
See also \cite[Section~2.2]{Shan} for the embedding of parabolic subgroups.

\begin{prop}
 \label{prop:twist-KZ}
 For a double coset representative $u \in {}^{(l,\mu)} W^{(m, \nu)}$ and
 a module $M \in \cO(W_{(k(u), \pi^{\sharp}(u))})$, 
 we have the following isomorphism of functors $:$
\begin{align*} 
\KZ_{W_{(k(u), \pi(u))}} \circ\, u (-) \cong T_u (-) \circ \KZ_{W_{(k(u), \pi^{\sharp}(u))}}. 
\end{align*}
\end{prop}

\begin{proof}
 Note that an $\sH_{(k(u), \pi(u))}$-module $T_u N$ for an $\sH_{(k(u), \pi^{\sharp}(u))}$-module
 $N$ is isomorphic to $N$ as a vector space by the map $N \longrightarrow T_u N$, $v \mapsto T_u v$
 and the action of $T_z \in \sH_{(k(u), \pi(u))}$ corresponding to $z \in W_{(k(u), \pi(u))}$ 
 on $T_u N$ is given by $T_z T_u v = T_u (T_{u^{-1} z u} v)$ for $v \in N$. For a module
 $M \in \cO(W_{(k(u), \pi(u))})$, we define a map
 \begin{align*}
  \kappa^{(u)}: (T_u(-) \circ \KZ_{W_{(k(u), \pi^{\sharp}(u))}})(M) &\longrightarrow
  (\KZ_{W_{(k(u), \pi(u))}} \circ\, \Theta^{(u)}_{W_{(k(u), \pi^{\sharp}(u))}})(M), \\
  T_u\, v &\mapsto \widehat{\Theta}^{(u)}_{W_{(k(u), \pi^{\sharp}(u))}} \circ
  S_{M}(u^{-1}({}^{u(p_0)} \widetilde{\omega}^{-1})) (v)
 \end{align*}
 for $v \in \KZ_{W_{(k(u), \pi^{\sharp}(u))}}(M) =
 (M^{an})_{p_0}^{\nabla}$. Here we remark
 that we have $\overline{\omega} = u$ and
 $u^{-1}({}^{u(p_0)}\widetilde{\omega}^{-1})$ is a path from $p_0$ to $u^{-1}(p_0)$.
 Obviously $\kappa^{(u)}$ is an isomorphism of $\C$-vector spaces. We see that the map 
 $\kappa^{(u)}$
 commutes with the action of the Hecke algebra $\sH_{(k(u), \pi(u))}$ by direct 
 computation as follows:
 For $v \in \KZ_{W_{(k(u), \pi^{\sharp}(u))}}(M) = (M^{an})_{p_0}^{\nabla}$ and
 $T_i \in \sH_{(k(u), \pi(u))}$ corresponding to the generator $s_i \in X_{(k(u), \pi(u))}$,
 we have
 \begin{multline*}
  \kappa^{(u)}(T_i (T_u v))
  = \kappa^{(u)}(T_u (T_{\psi(i)} v)) \\
  = \widehat{\Theta}^{(u)}_{W_{(k(u), \pi^{\sharp}(u))}} \circ 
  S_{M}(u^{-1}({}^{u(p_0)} \widetilde{\omega}^{-1})) (S_M({}^{p_0} \widetilde{\sigma}_{\psi(i)}^{-1}) s_{\psi(i)} v)) \\
  = S_{u M}({}^{u(p_0)}\widetilde{\omega}^{-1} \cdot u({}^{p_0} \widetilde{\sigma}_{\psi(i)}^{-1}))
  s_i \widehat{\Theta}^{(u)}_{W_{(k(u), \pi^{\sharp}(u))}} (v).
 \end{multline*}
 Here we can deduce that
 the path 
 ${}^{u(p_0)}\widetilde{\omega}^{-1} \cdot u({}^{p_0} \widetilde{\sigma}_{\psi(i)}^{-1}) = (u({}^{p_0} \widetilde{\sigma}_{\psi(i)}) \cdot {}^{u(p_0)} \widetilde{\omega})^{-1} = ({}^{(u s_{\psi(i)})(p_0)} \widetilde{\sigma_{\psi(i)} \cdot \omega})^{-1}$ 
 is lifted from
 the element $(\sigma_{\psi(i)} \cdot \omega)^{-1} \in B_{n,r}$, being equal to
 $(\omega \cdot \sigma_{i})^{-1} \in B_{n,r}$ by \eqref{eq:braid us_j =
 s_j'u}. Thus, by the uniqueness of lifting from the fixed initial base point
 $(u s_{\psi(i)})(p_0) = (s_i u)(p_0)$, we have
\[
S_{uM}(({}^{(us_{\psi(i)})(p_0)} \widetilde{\sigma_{\psi(i)} \cdot \omega})^{-1}) = S_{uM}(({}^{(s_i u)(p_0)} \widetilde{\omega \cdot \sigma_i})^{-1}) = S_{uM}({}^{p_0}\widetilde{\sigma}_i^{-1} \cdot s_i({}^{u(p_0)}\widetilde{\omega}^{-1})). 
\]
Therefore, we have
 \begin{multline*}
  \kappa^{(u)}(T_i (T_u v))
  = S_{{uM}}({}^{u(p_0)}\widetilde{\omega}^{-1} \cdot u({}^{p_0} \widetilde{\sigma}_{\psi(i)}^{-1}))
  s_i \widehat{\Theta}^{(u)}_{(k(u), \pi^{\sharp}(u))} (v) \\
  = S_{uM}({}^{p_0}\widetilde{\sigma}_i^{-1} \cdot s_i({}^{u(p_0)}\widetilde{\omega}^{-1}))
  s_i \widehat{\Theta}^{(u)}_{(k(u), \pi^{\sharp}(u))} (v) \\
  = (S_{uM}({}^{p_0} \widetilde{\sigma}_i^{-1}) s_i) (S_{uM}({}^{u(p_0)} \widetilde{\omega}^{-1}) \widehat{\Theta}^{(u)}_{W_{(k(u), \pi^{\sharp}(u))}} (v))
  = T_i \cdot \kappa^{(u)}(T_u v).
 \end{multline*}
 That is, the map $\kappa^{(u)}$ is a homomorphism of $\sH_{(k(u), \pi(u))}$-modules.
 It is clear from the definition that $\kappa^{(u)}$ is functorial, and hence we have
 the desired isomorphism of functors.
\end{proof}


\begin{prop}
\label{prop:Mackey-RCA}
 We have the following isomorphism of functors $:$
 \[
  \ORes^{W_{n,r}}_{W_{(l, \mu)}} \circ \OInd^{W_{n,r}}_{W_{(m, \nu)}}
 \cong \bigoplus_{u \in {}^{(l, \mu)} W^{(m, \nu)}} 
 \OInd^{W_{(l, \mu)}}_{W_{(k(u), \pi(u))}} \circ\, u (-) \circ
 \ORes^{W_{(m, \nu)}}_{W_{(k(u), \pi^{\sharp}(u))}}.
 \]
\end{prop}
\begin{proof}
Combining   
 Theorem \ref{H Thm Mackey} (\roiii) and
 Proposition \ref{prop:twist-KZ} with
Proposition~\ref{prop:H-Mackey-RCA}, we deduce the desired result.
 \end{proof}


\para
In this subsection, we shall polish Proposition~\ref{prop:Mackey-RCA} to the
Mackey formula in $\cO$ for all the parabolic subgroups.
Note that 
any parabolic subgroup of $W_{n,r}$ coincides with 
$x W_{(l,\mu)} x^{-1}$ for some $l \geq 0$, $\mu \vDash n-l$ and $x \in W_{n,r}$. 
By applying the twisting functor $\Theta^{(x)}_{W_{(l,\mu)}}$ defined in
the paragraph
\ref{twist functor} to Proposition \ref{prop:Mackey-RCA}, 
we finally obtain the second main theorem of this paper, which supports
Conjecture~\ref{ourconj} : 
\begin{thm}[The Mackey formula for $\cO$ over cyclotomic rational Cherednik algebras]\label{thm:Mackey-RCA}
Let $W_a, W_b$ be parabolic subgroups of $W_{n,r}$, 
and $^a W^b$ be a complete set of double coset representatives of $W_a \backslash W_{n,r} / W_b$. 
Then we have the following isomorphism of functors $:$ 
\begin{align*}
  \ORes^W_{W_a} \circ \OInd^W_{W_b}
 \cong \bigoplus_{u \in {}^{a} W^{b}} 
\OInd^{W_a}_{W_a \cap u W_b u^{-1}} \circ\, u (-) \circ
 \ORes^{W_b}_{u^{-1} W_a u \cap W_b }.
\end{align*}
\end{thm}
\begin{proof}
Let $x, y \in W$, $l, m \geq 0$, $\mu \vDash n-l$ and $\nu \vDash n-m$ be such that 
$W_a = x W_{(l,\mu)} x^{-1}$ and $W_b = y W_{(m,\nu)} y^{-1}$. 
By Proposition \ref{prop:Mackey-RCA}, 
we have 
\begin{align*}
&\Theta^{(x)}_{W_{(l,\mu)}} \circ \ORes^{W_{n,r}}_{W_{(l, \mu)}} 
\circ \OInd^{W_{n,r}}_{W_{(m, \nu)}} \circ\, \Theta^{(y^{-1})}_{y W_{(m,\nu)} y^{-1}} 
\\
&\cong \bigoplus_{u \in {}^{(l, \mu)} W^{(m, \nu)}} 
 \Theta^{(x)}_{W_{(l,\mu)}} \circ \OInd^{W_{(l, \mu)}}_{W_{(k(u), \pi(u))}} \circ\, u (-) \circ
 \ORes^{W_{(m, \nu)}}_{W_{(k(u), \pi^{\sharp}(u))}} \circ\, \Theta^{(y^{-1})}_{y W_{(m,\nu)} y^{-1}} .
\end{align*}
Applying Proposition \ref{Prop Theta} to this isomorphism, 
we have 
\begin{align*}
&\ORes^{W_{n,r}}_{x W_{(l,\mu)} x^{-1}} \circ \OInd^{W_{n,r}}_{y W_{(m,\nu)} y^{-1}} 
\\
&\cong 
\bigoplus_{u \in {}^{(l, \mu)} W^{(m, \nu)}} 
\OInd^{ x W_{(l,\mu)} x^{-1}}_{x W_{(k(u), \pi(u))} x^{-1}} 
\circ\, x u y^{-1} (-) \circ \ORes^{y W_{(m,\nu)} y^{-1}}_{y W_{(k(u), \pi^{\sharp}(u))}y^{-1}}. 
\end{align*}
By direct calculation, we can check that 
$^{(l,\mu)} W^{(m,\nu)}_{(x,y)} =\{ x u y^{-1} \mid u \in \,^{(l,\mu)} W^{(m,\nu)}\}$  
is a complete set of double coset representatives of $W_a \backslash W_{n,r}/ W_b$. 
We can also check that 
\begin{align*}
&x W_{(k(u), \pi(u))} x^{-1} = W_a \cap x u y^{-1} W_b (x u y^{-1})^{-1} 
\\
&y W_{(k(u), \pi^{\sharp}(u))}y^{-1} = (x u y^{-1})^{-1} W_a (x u y^{-1}) \cap W_b.
\end{align*} 
Thus we have 
\begin{align*}
  \ORes^W_{W_a} \circ \OInd^W_{W_b}
 \cong \bigoplus_{u \in {}^{(l,\mu)} W^{(m,\nu)}_{(x,y)}} 
\OInd^{W_a}_{W_a \cap u W_b u^{-1}} \circ\, u (-) \circ
 \ORes^{W_b}_{u^{-1} W_a u \cap W_b }.
\end{align*}

Since the isomorphism $\theta^{(x)}_{W'}$ is a $\C$-algebra automorphism of $H(W')$ if
the element $x$ belongs to the subgroup $W'$, the functor $u(-)$ does not depend on
the choice of the double coset representative $u$. Thus we have the
 desired isomorphism.
\end{proof}



\appendix

\section{Lifting an isomorphism of functors through the double centralizer properties} 

\para 
Let $\bk$ be a field. 
Let $\fC_i$ ($i=1,2$) be an abelian $\bk$-linear category in which every object has finite length. 
We assume that $\fC_i$ has enough projective objects. 
Let $\fC_i \proj$  be the full subcategory of $\fC_i$ 
consisting of projective objects, 
and $I_i : \fC_i \proj \ra \fC_i$ 
be the canonical embedding functor. 
Then we have the following lemma. 


\begin{lem}[{cf. \cite[Lemma 1.2]{Shan}}]
\label{Lemma F I1 G I2}
Let $F$ and $G$ be right exact functors from $\fC_1$ to $\fC_2$. 
If there exists an isomorphism of functors $F \circ I_1 \cong G \circ I_1$, 
we have $F \cong G$. 
\end{lem}
\begin{proof} 
We can prove the lemma by the same way as the proof of \cite[Lemma 1.2]{Shan}. 
\end{proof}

\para 
Let $P_i$ ($i=1,2$) be projective objects in $\fC_i$. 
Put $\fD_i = B_i \cmod $, 
where $B_i = \End_{\fC_i}(P_i)^{\operatorname{opp}}$. 
We consider the functor 
$\Omega_i = \Hom_{\fC_i} (P_i, -) : \fC_i \ra \fD_i$. 
Let $\Phi_i : \fD_i \ra \fC_i$ be the right adjoint functor of $\Omega_i$. 
We denote by $\ve_i : \Omega_i \circ \Phi_i \ra \Id_{\fD_i}$ 
(resp. $\eta_i : \Id_{\fC_i} \ra \Phi_i \circ \Omega_i$) 
the corresponding unit (resp. counit). 
Then we have the following lemma. 


\begin{lem}[{\cite[4.2.1 and Proposition 4.33]{Rouq}}] 
\label{Lemma double}
Assume that the restriction of  $\Omega_i$ ($i=1,2$) 
to $\fC_i \proj$ 
is fully faithful. 
Then we have the following. 
\begin{enumerate} 
\item 
The unit $\ve_i $  implies the isomorphism of functors 
$\Omega_i \circ \Phi_i \cong \Id_{\fD_i}$.  

\item 
The counit $\eta_i$ implies the isomorphism of functors 
$\Id_{\fC_i} \circ I_i \cong \Phi_i \circ \Omega_i \circ I_i$. 

\end{enumerate}
\end{lem}

Under the setting in the above, 
we have the following lifting of isomorphism of functors. 


\begin{prop}
\label{Prop lifting}
Let $F, G : \fC_1 \ra \fC_2$ (resp. $F', G' : \fD_1 \ra \fD_2$) be functors. 
Assume that 
\begin{enumerate} 
\item the restriction of  $\Omega_i$ ($i=1,2$) 
to $\fC_i \proj$ is fully faithful, 

\item 
$F$ and $G$ are right exact, 

\item 
$F$ and $G$ map projective objects to projective ones, 

\item 
the isomorphisms of functors 
$\Omega_2 \circ F \cong F' \circ \Omega_1$ and $\Omega_2 \circ G \cong G' \circ \Omega_1$.
\end{enumerate} 
Then, 
if there exists an isomorphism of functors 
$F' \cong G'$, 
we have $F \cong G$. 
\end{prop}

\begin{proof}
The isomorphism $F' \cong G'$ implies 
the isomorphism 
$F' \circ \Omega_1 \cong G' \circ \Omega_1$. 
By the assumption (\roiv), 
we have 
$\Omega_2 \circ F \cong \Omega_2 \circ G$. 
This implies that 
$\Phi_2 \circ \Omega_2 \circ F \circ I_1 \cong \Phi_2 \circ \Omega_2 \circ G \circ I_1$. 
By the assumption (\roiii), 
we have 
\begin{align} 
\label{iso Phi Omega I F I}
\Phi_2 \circ \Omega_2 \circ I_2 \circ F \circ I_1 \cong \Phi_2 \circ \Omega_2 \circ I_2 \circ G \circ I_1.
\end{align}  
On the other hand, 
by the assumption (\roi) together with Lemma \ref{Lemma double}, 
we have $\Phi_2 \circ \Omega_2 \circ I_2 \cong \Id_2 \circ I_2$. 
Thus, \eqref{iso Phi Omega I F I} implies 
$F \circ I_1 \cong G \circ I_1$, 
and we have 
$F \cong G$ by the assumption (\roii) together with Lemma \ref{Lemma F I1 G I2}. 
\end{proof}



\section{A root system for $G(r,1,n)$} 

In this appendix, 
we explain some connection with a root system for the complex reflection group of type $G(r,1,n)$
introduced in \cite{BM}. 
We use notation and results given in \cite{RS}. 

\para 
Let $V$ be a complex vector space with a basis $\{\epsilon_1, \epsilon_2,\dots, \epsilon_n\}$. 
Let $\zeta = \operatorname{exp}(2 \pi \sqrt{-1}/r)$ 
be the primitive $r$-th root of unity. 
Then $W_{n,r}$ acts on $V$ by 
\begin{align*}
(x t_1^{a_1} \dots t_n^{a_n}) \cdot \epsilon_i = \zeta^{a_i} \epsilon_{x(i)} 
\quad (x \in \fS_n, \, 0 \leq a_1, \dots, a_n \leq r-1, \, 1\leq i  \leq n). 
\end{align*}
For $i=1,\dots, n-1$, 
a vector $\epsilon_{i+1} - \epsilon_i$ is orthogonal to the reflection hyperplane corresponding the reflection $s_i$, 
and a vector $\epsilon_1$ is orthogonal to the reflection hyperplane corresponding the reflection $s_0$. 
Put 
\begin{align*}
\overline{\Delta} = \{\epsilon_{i+1} - \epsilon_i \mid 1\leq i \leq n-1\} \cup \{\epsilon_1\} 
\end{align*}
and put $\overline{\Phi}  =W_{n,r} \cdot \overline{\Delta} $. 
Then we have 
\begin{align*}
\overline{\Phi} 
= \{ \zeta^{a} \epsilon_{i} - \zeta^{b} \epsilon_j \mid 1\leq i\not=j \leq n, \, 0 \leq a, b \leq r-1\} 
	\cup \{ \zeta^{a} \epsilon_i \mid 1\leq i \leq n, \, 0\leq a \leq r-1\}.
\end{align*}


\para 
In this appendix, 
we identify elements $a \in \ZZ / r \ZZ$ with integers $0 \leq a \leq r-1$. 
We consider a set 
$X=\{ e_i^{(a)} \mid 1 \leq i \leq n, \, a \in \ZZ /  r \ZZ\}$, 
where $e_i^{(a)}$ is just a symbol indexed by $i$ and $a$. 
One can define an action of $W_{n,r}$ on $X$ by 
\begin{align*}
(x t_1^{a_1} \dots t_n^{a_n}) \cdot e_i^{(a)} = e_{x(i)}^{(a + a_i)}
\quad (x \in \fS_n, \, 0 \leq a_1, \dots, a_n \leq r-1, \, 1\leq i  \leq n, \, a \in \ZZ / r \ZZ). 
\end{align*}
We also define another action of $W_{n,r}$ on $X$ by 
\begin{align*}
(x t_1^{a_1} \dots t_n^{a_n}) \ast e_i^{(a)} = e_{x(i)}^{(a - a_i)}
\quad (x \in \fS_n, \, 0 \leq a_1, \dots, a_n \leq r-1, \, 1\leq i  \leq n, \, a \in \ZZ / r \ZZ). 
\end{align*}

We express an element $(e_i^{(a)}, e_j^{(b)}) \in X \times X$ as $e_i^{(a)} - e_j^{(b)}$ in the case where $i \not=j$. 
Then we define a root system $\Phi$ for $W_{n,r}$ by 
\begin{align*}
\Phi = \{e_i^{(a)} - e_j^{(b)} \mid 1 \leq i \not=j \leq n, \, a, b \in \ZZ / r \ZZ\} 
	\cup \{ e_i^{(a)} \mid 1 \leq i \leq n, \, a \in \ZZ / r \ZZ\}. 
\end{align*}
We define subsets $\Phi_0, \Omega$ and $\Delta$  of $\Phi$ by 
\begin{align*}
&\Phi_0 = \{ e_i^{(a)} - e_j^{(b)} \in \Phi \mid i > j, \, a =0\} \cup \{ e_i^{(a)} - e_j^{(b)} \mid i < j, \, b \not= 0\} 
	\cup \{e_i^{(0)} \mid 1 \leq i \leq n\}, 
\\
&\Omega = \{ e_i^{(0)} - e_j^{(b)} \mid 1 \leq j < i \leq n, \, b \in \ZZ / r \ZZ\} 
	\cup \{e_i^{(0)} \mid 1 \leq i \leq n\}, 
\\
& \Delta = \{ e_{i+1}^{(0)} - e_i^{(0)} \mid 1\leq i \leq n-1\} \cup \{ e_1^{(0)}\}
\end{align*}

Let $\varphi : \Phi \ra V$ be a map such that 
$\varphi(e_i^{(a)} - e_j^{(b)}) = \zeta^{a} \epsilon_i - \zeta^{b} \epsilon_j$ 
and 
$\varphi( e_i^{(a)}) = \zeta^a \epsilon_i$, 
then we see that 
$\varphi (\Phi) = \overline{\Phi}$ and $\varphi(\Delta)= \overline{\Delta}$. 


\remark 
(\roi). 
In the case where $r=2$ (in this case, $W_{n,2}$ coincides with the Weyl group of type $B_n$), 
we see that 
$\overline{\Phi}$ (resp. $\overline{\Delta}$) 
coincides with a root system (resp. a set of simple roots) 
for the  Weyl group of type $B_n$. 
Moreover, 
$\varphi(\Omega)$ coincides with the set of positive roots with respect to $\overline{\Delta}$ 
of the Weyl group of type $B_n$. 

(\roii). 
In general case, 
$\Omega$ is not a positive root in the sense of \cite{BM}, 
but $\Omega$ plays the role of positive roots. 
Moreover, in this appendix, we follow notion in \cite{RS}, 
and the definitions of $\Phi$ and $\Omega$ are different from 
them in \cite{BM}. 
See \cite[Remark 1.4]{RS} for these differences. 


\para 
For $0 \leq l \leq n$ and $\mu \vDash n-l$, 
we obtain the root system $\Phi_{(l,\mu)}$ and subsets 
$\Omega_{(l,\mu)}, \Delta_{(l,\mu)} \subset \Phi_{(l,\mu)}$ 
for the parabolic subgroup $W_{(l,\mu)}$ of $W_{n,r}$ by 
\begin{align*}
\Phi_{(l,\mu)} 
&= \{ e_i^{(a)} - e_j^{(b)} \mid 1 \leq i \not= j \leq l, \, a,b \in \ZZ / r \ZZ \} 
	\cup \{ e_i^{(a)} \mid 1\leq i \leq l, \, a \in \ZZ / r \ZZ \}
	\\ & \qquad 
{\cup} \bigcup_{p=1}^{\ell (\mu)} \{ e_i^{(0)} - e_j^{(0)} \mid  
		l + |\mu|_{p-1} +1 \leq i \not= j \leq l +  |\mu|_p \}, 
\\
\Omega_{(l,\mu)} 
&= \{ e_i^{(0)} - e_j^{(b)} \mid 1 \leq j <i \leq l, \, b \in \ZZ / r \ZZ\} \cup \{e_i^{(0)} \mid 1 \leq i \leq l \}
	\\ & \qquad 
{\cup} \bigcup_{p=1}^{\ell (\mu)} \{ e_i^{(0)} - e_j^{(0)} \mid  
		l + |\mu|_{p-1} +1 \leq j < i  \leq l + |\mu|_p \}, 
\\
\Delta_{(l,\mu)} 
&= \{ e_{i+1}^{(0)} - e_i^{(0)} \mid 1 \leq i \leq l-1 \} \cup \{e_1^{(0)} \mid l \not=0 \}
	\\ & \qquad 
{\cup} \bigcup_{p=1}^{\ell (\mu)} \{ e_{i+1}^{(0)} - e_i^{(0)} \mid  
		l + |\mu|_{p-1} +1 \leq i  \leq l + |\mu|_p -1 \}, 
\end{align*}
where we put $|\mu|_p = \sum_{k=1}^p \mu_k$ with $|\mu|_0=0$, 
and $\{ e_1^{(0)} \mid l \not=0\} = \begin{cases} \{e_i^{(0)}\} &\text{ if } l \not=0, 
	\\ \emptyset &\text{ if } l=0. \end{cases}$ 
We also define 
\begin{align*}
\wt{\Omega}_{(l,\mu)}
&= \{ e_i^{(0)} - e_j^{(b)} \mid 1 \leq j <i \leq l, \, b \in \ZZ / r \ZZ\} \cup \{e_i^{(0)} \mid 1 \leq i \leq l \}
	\\ & \qquad 
{\cup} \bigcup_{p=1}^{\ell (\mu)} \{ e_i^{(0)} - e_j^{(b)} \mid  
		l + |\mu|_{p-1} +1 \leq j < i  \leq l + |\mu|_p, \, b \in \ZZ / r \ZZ  \}. 
\end{align*}
Then we have 
$\Omega_{(l,\mu)} \subset \wt{\Omega}_{(l,\mu)}$, 
and we also have
$\Omega_{(l,\mu)} = \wt{\Omega}_{(l,\mu)}$ 
if and only if 
$\mu=(1^{n-l})$. 


\para 
For $w \in W_{n,r}$, 
let $\ell (w)$ be the smallest number $k$ such that 
$w$ is expressed as a product $w = s_{i_1} \dots s_{i_k}$ ($s_{i_j} \in \{s_0,s_1,\dots, s_{n-1}\}$). 

For $0 \leq l \leq n$ and $\mu \vDash n-l$,
we define a subsets $\mathcal{R}_{(l,\mu)}$, $\mathcal{R}_{(l,\mu)}^{\ast}$, 
$\mathcal{R}_{(l,\mu)}^0$ and $\mathcal{R}_{(l,\mu)}^{\ast 0}$ 
of $W_{n,r}$ by 
\begin{align*}
& \mathcal{R}_{(l,\mu)} = \{ w \in W \mid w (\Omega_{(l,\mu)}) \subset \Phi_0\}, 
	& \mathcal{R}_{(l,\mu)}^{\ast} = \{ w \in W \mid w^{-1} \ast \Omega_{(l,\mu)} \subset \Phi_0\}, 
\\
& \mathcal{R}_{(l,\mu)}^0 = \{ w \in W \mid w (\wt{\Omega}_{(l,\mu)}) \subset \Phi_0\}, 
	& \mathcal{R}_{(l,\mu)}^{\ast 0} = \{ w \in W \mid w^{-1} \ast \wt{\Omega}_{(l,\mu)} \subset \Phi_0\}. 
\end{align*}
Since $\Omega_{(l,\mu)} \subset \wt{\Omega}_{(l,\mu)}$, 
we have 
$\mathcal{R}_{(l,\mu)}^0 \subset \mathcal{R}_{(l,\mu)}$ 
(resp. $\mathcal{R}_{(l,\mu)}^{\ast 0} \subset \mathcal{R}_{(l,\mu)}^{\ast}$). 
The following proposition is proven in \cite[Lemma 1.27, Proposition 1.28, Corollary 1.29]{RS}. 

\begin{prop} 
\label{R Prop}
For $0 \leq l \leq n$ and $\mu \vDash n-l$, we have the following: 
\begin{enumerate} 
\item 
\begin{enumerate} 
\item 
For $w \in W_{n,r}$, 
we have 
$w (\Omega_{(l,\mu)}) \subset \Phi_0$ if and only if $w (\Delta_{(l,\mu)}) \subset \Phi_0$. 
\item 
For $w \in W_{n,r}$, 
we have 
$w^{-1} \ast \Omega_{(l,\mu)} \subset \Phi_0 $ if and only if $w^{-1} \ast \Delta_{(l,\mu)} \subset \Phi_0$.  
\end{enumerate} 

\item 
\begin{enumerate} 
\item 
For $w \in \mathcal{R}_{(l,\mu)}^0$ and $ w' \in W_{(l,\mu)}$, 
we have $\ell(w w') = \ell(w) + \ell(w')$. 
\item
For $w \in \mathcal{R}_{(l,\mu)}^{\ast}$ and $w' \in W_{(l,\mu)}$, 
we have $\ell(w' w) = \ell(w') + \ell (w)$. 
\end{enumerate}

\item 
\begin{enumerate} 
\item 
For $w \in W_{n,r}$, 
if $\ell (w)$ is minimal among all elements in $w W_{(l,\mu)}$, 
we have $w \in \mathcal{R}_{(l,\mu)}$. 

\item 
For $w \in W_{n,r}$, 
if $\ell (w)$ is minimal among all elements in $W_{(l,\mu)} w$, 
we have $w \in \mathcal{R}_{(l,\mu)}^{\ast}$. 
\end{enumerate}

\item 
\begin{enumerate} 
\item 
In the case where $\mathcal{R}_{(l,\mu)} = \mathcal{R}_{(l,\mu)}^0$, 
the set $\mathcal{R}_{(l,\mu)}$ is a complete set of coset representatives for $W_{n,r}/ W_{(l,\mu)}$. 

\item 
In the case where $\mathcal{R}_{(l,\mu)}^{\ast} = \mathcal{R}_{(l,\mu)}^{\ast 0}$, 
the set $\mathcal{R}_{(l,\mu)}^{\ast}$ is a complete set of coset representatives for $W_{(l,\mu)} \backslash W_{n,r} $. 

\end{enumerate}
\end{enumerate}
\end{prop}


\para 
Assume that 
$l \not=0$ and 
$\mu= (1^{n-l})$. 
In this case, we have $W_{(l,\mu)}= W_{l,r}$, 
$\Omega_{(l,\mu)} = \wt{\Omega}_{(l,\mu)}$, 
$\mathcal{R}_{(l,\mu)} = \mathcal{R}_{(l,\mu)}^0 $ 
and 
$\Delta_{(l,\mu)} = \{e_{i+1}^{(0)} - e_i^{(0)} \mid 1 \leq i \leq l-1\} \cup \{e_1^{(0)}\}$. 
For $x t_{l+1}^{a_{l+1}} t_{l+2}^{a_{l+2}} \dots t_n^{a_n} \in W^{(l,\mu)}$ 
and $e_{i+1}^{(0)} - e_i^{(0)}$ ($1 \leq i \leq l-1$), 
we have 
\begin{align*}
(x t_{l+1}^{a_{l+1}} t_{l+2}^{a_{l+2}} \dots t_n^{a_n} ) \cdot (e_{i+1}^{(0)} - e_i^{(0)}) 
= e_{x(i+1)}^{(0)} - e_{x(i)}^{(0)}, 
\end{align*}
and 
$x(i+1) > x(i)$ since $x \in \fS^{(l,\mu)}$ and $s_i \in S_{(l,\mu)}$. 
We also have 
\begin{align*} 
(x t_{l+1}^{a_{l+1}} t_{l+2}^{a_{l+2}} \dots t_n^{a_n} ) \cdot e_1^{(0)} = e_{x(1)}^{(0)}.
\end{align*}
Thus we see that 
$(x t_{l+1}^{a_{l+1}} t_{l+2}^{a_{l+2}} \dots t_n^{a_n} ) (\Delta_{(l,\mu)}) \subset \Phi_0$ 
for any $x t_{l+1}^{a_{l+1}} t_{l+2}^{a_{l+2}} \dots t_n^{a_n}  \in W^{(l,\mu)}$. 
Then, by Proposition \ref{R Prop} (\roi), 
we have that 
$W^{(l,\mu)} \subset \mathcal{R}_{(l,\mu)}$. 
On the other hand, 
$W^{(l,\mu)}$ (resp. $\mathcal{R}_{(l,\mu)}$ ) 
is a complete set of representatives for $W/ W_{(l,\mu)}$ 
by Lemma \ref{W Lemma coset rep} (resp. Proposition \ref{R Prop} (\roiv)). 
Thus we have $W^{(l,\mu)} = \mathcal{R}_{(l,\mu)}$ if $\mu=(1^{n-l})$. 
Similarly, we have $^{(l,\mu)} W = \mathcal{R}_{(l,\mu)}^{\ast}$ if $\mu=(1^{n-l})$. 
Moreover we have 
$^{(l,\mu)} W^{(m,\nu)} = ^{(l,\mu)} W \cap W^{(m,\nu)} $ if $\nu =(1^{n-m})$ by Lemma \ref{W Lemma nu 1n-m}. 
As a consequence, 
we have the following corollary.

\begin{cor} 
Assume that $\mu=(1^{n-l})$ and $\nu =(1^{n-m})$. 
Then we have 
$^{(l,\mu)} W = \mathcal{R}_{(l,\mu)}^{\ast}$ and $W^{(m,\nu)} = \mathcal{R}_{(m,\nu)}$. 
Moreover, 
$\mathcal{R}_{(l,\mu)}^{\ast}  \cap \mathcal{R}_{(m,\nu)} = \,^{(l,\mu)} W \cap W^{(m,\nu)}$ 
is a complete set of representatives for $W_{(l,\mu)} \backslash W_{n,r} / W_{(m,\nu)}$. 
\end{cor}


\remark 
\label{Remark root system}
(\roi). 
In the case where $\mu \not= (1^{n-l})$, 
there exists $i >l$ such that $e_{i+1}^{(0)} - e_i^{(0)} \in \Delta_{(l, \mu)}$. 
For $e_{i+1}^{(0)} - e_i^{(0)} \in \Delta_{(l,\mu)}$ such that $i >l$ and  
$x t_{l+1}^{a_{l+1}} \dots t_n^{a_n} \in W^{(l,\mu)}$, 
we have 
\begin{align*}
(x t_{l+1}^{a_{l+1}} \dots t_n^{a_n}) \cdot (e_{i+1}^{(0)} - e_i^{(0)})
= e_{x(i+1)}^{(a_{i+1})} - e_{x(i)}^{(a_i)},  
\end{align*}
and $x(i+1) > x(i)$ since $x \in \fS^{(l,\mu)}$ and $s_i \in S_{(l,\mu)}$. 
Moreover, 
$e_{x(i+1)}^{(a_{i+1})} - e_{x(i)}^{(a_i)}  \not\in \Phi_0$ if $a_{i+1} \not=0$. 
Thus, we see that $W^{(l,\mu)} \not\subset \mathcal{R}_{(l,\mu)}$ if $\mu \not=(1^{n-l})$. 
Similarly, we have $^{(l,\mu)} W \not\subset \mathcal{R}_{(l,\mu)}^{\ast}$ if $\mu \not=(1^{n-l})$. 

(\roii). 
In general case, 
we do not know if we can characterize the set $W^{(l,\mu)}$ 
(or another complete set of representatives for $W/ W_{(l,\mu)}$) 
by using the root system $\Phi$. 



\bibliographystyle{halpha} 
\bibliography{refs}

\begin{thebibliography}{DDPW08}

\bibitem[AK94]{AK}
Susumu Ariki and Kazuhiko Koike.
\newblock A {H}ecke algebra of {$({\bf Z}/r{\bf Z})\wr{\mathfrak{S}}_n$} and
  construction of its irreducible representations.
\newblock {\em Adv. Math.}, 106(2):216--243, 1994.

\bibitem[BE09]{BE}
Roman Bezrukavnikov and Pavel Etingof.
\newblock Parabolic induction and restriction functors for rational {C}herednik
  algebras.
\newblock {\em Selecta Math. (N.S.)}, 14(3-4):397--425, 2009.

\bibitem[BM97]{BM}
Kirsten Bremke and Gunter Malle.
\newblock Reduced words and a length function for {$G(e,1,n)$}.
\newblock {\em Indag. Math. (N.S.)}, 8(4):453--469, 1997.

\bibitem[BMR98]{BMR}
Michel Brou\'e, Gunter Malle, and Rapha\"el Rouquier.
\newblock Complex reflection groups, braid groups, {H}ecke algebras.
\newblock {\em J. Reine Angew. Math.}, 500:127--190, 1998.

\bibitem[Bon00]{Bon00}
C\'edric Bonnaf\'e.
\newblock Mackey formula in type {A}.
\newblock {\em Proc. London Math. Soc. (3)}, 80(3):545--574, 2000.

\bibitem[CR90]{CR}
Charles~W. Curtis and Irving Reiner.
\newblock {\em Methods of representation theory. {V}ol. {I}}.
\newblock Wiley Classics Library. John Wiley \& Sons, Inc., New York, 1990.
\newblock With applications to finite groups and orders, Reprint of the 1981
  original, A Wiley-Interscience Publication.

\bibitem[DDPW08]{DDPW}
Bangming Deng, Jie Du, Brian Parshall, and Jianpan Wang.
\newblock {\em Finite dimensional algebras and quantum groups}, volume 150 of
  {\em Mathematical Surveys and Monographs}.
\newblock American Mathematical Society, Providence, RI, 2008.

\bibitem[DF92]{MR1179779}
Richard Dipper and Peter Fleischmann.
\newblock Modular {H}arish-{C}handra theory. {I}.
\newblock {\em Math. Z.}, 211(1):49--71, 1992.

\bibitem[DF94]{MR1249581}
Richard Dipper and Peter Fleischmann.
\newblock Modular {H}arish-{C}handra theory. {II}.
\newblock {\em Arch. Math. (Basel)}, 62(1):26--32, 1994.

\bibitem[GGOR03]{GGOR}
Victor Ginzburg, Nicolas Guay, Eric Opdam, and Rapha\"el Rouquier.
\newblock On the category {$\mathscr{O}$} for rational {C}herednik algebras.
\newblock {\em Invent. Math.}, 154(3):617--651, 2003.

\bibitem[HL80]{MR570873}
R.~B. Howlett and G.~I. Lehrer.
\newblock Induced cuspidal representations and generalised {H}ecke rings.
\newblock {\em Invent. Math.}, 58(1):37--64, 1980.

\bibitem[HL83]{MR716849}
R.~B. Howlett and G.~I. Lehrer.
\newblock Representations of generic algebras and finite groups of {L}ie type.
\newblock {\em Trans. Amer. Math. Soc.}, 280(2):753--779, 1983.

\bibitem[Hum90]{H}
James~E. Humphreys.
\newblock {\em Reflection groups and {C}oxeter groups}, volume~29 of {\em
  Cambridge Studies in Advanced Mathematics}.
\newblock Cambridge University Press, Cambridge, 1990.

\bibitem[Jon90]{MR948191}
Lenny~K. Jones.
\newblock Centers of generic {H}ecke algebras.
\newblock {\em Trans. Amer. Math. Soc.}, 317(1):361--392, 1990.

\bibitem[Los13]{Losev}
Ivan Losev.
\newblock On isomorphisms of certain functors for {C}herednik algebras.
\newblock {\em Represent. Theory}, 17:247--262, 2013.

\bibitem[LSA18]{LS}
Ivan Losev and Seth Shelley-Abrahamson.
\newblock On refined filtration by supports for rational cherednik categories
  $\cO$.
\newblock {\em Sel. Math. (N.S.)}, published online,
\texttt{https://doi.org/10.1007/s00029-018-0390-6}, 2018.

\bibitem[Mac51]{Mac}
George~W. Mackey.
\newblock On induced representations of groups.
\newblock {\em Amer. J. Math.}, 73:576--592, 1951.

\bibitem[Mat86]{Mats}
Hideyuki Matsumura.
\newblock {\em Commutative ring theory}, volume~8 of {\em Cambridge Studies in
  Advanced Mathematics}.
\newblock Cambridge University Press, Cambridge, 1986.
\newblock Translated from the Japanese by M. Reid.

\bibitem[Mat99]{M}
Andrew Mathas.
\newblock {\em Iwahori-{H}ecke algebras and {S}chur algebras of the symmetric
  group}, volume~15 of {\em University Lecture Series}.
\newblock American Mathematical Society, Providence, RI, 1999.

\bibitem[Rou08]{Rouq}
Rapha\"el Rouquier.
\newblock {$q$}-{S}chur algebras and complex reflection groups.
\newblock {\em Mosc. Math. J.}, 8(1):119--158, 184, 2008.

\bibitem[RS98]{RS}
Konstantinos Rampetas and Toshiaki Shoji.
\newblock Length functions and {D}emazure operators for {$G(e,1,n)$}. {I},
  {II}.
\newblock {\em Indag. Math. (N.S.)}, 9(4):563--580, 581--594, 1998.

\bibitem[Sha11]{Shan}
Peng Shan.
\newblock Crystals of {F}ock spaces and cyclotomic rational double affine
  {H}ecke algebras.
\newblock {\em Ann. Sci. \'Ec. Norm. Sup\'er. (4)}, 44(1):147--182, 2011.

\bibitem[SV12]{SV}
P.~Shan and E.~Vasserot.
\newblock Heisenberg algebras and rational double affine {H}ecke algebras.
\newblock {\em J. Amer. Math. Soc.}, 25(4):959--1031, 2012.

\bibitem[Vaz02]{MR1925135}
M.~Vazirani.
\newblock Filtrations on the {M}ackey decomposition for cyclotomic {H}ecke
  algebras.
\newblock {\em J. Algebra}, 252(2):205--227, 2002.

\bibitem[Wei94]{Weibel}
Charles~A. Weibel.
\newblock {\em An introduction to homological algebra}, volume~38 of {\em
  Cambridge Studies in Advanced Mathematics}.
\newblock Cambridge University Press, Cambridge, 1994.

\end{thebibliography}
\end{document}